\documentclass[a4paper, reqno]{amsart} 
\usepackage{graphicx} 

\usepackage{amsmath,amsthm, amssymb, amsfonts}
\usepackage[margin=3cm]{geometry} 
\usepackage{url}

\newtheorem{theorem}{Theorem}[section]
\newtheorem{corollary}[theorem]{Corollary}
\newtheorem{lemma}[theorem]{Lemma}
\newtheorem{proposition}[theorem]{Proposition}
\theoremstyle{definition}
\newtheorem{definition}[theorem]{Definition}

\newtheorem{problem}[]{Problem}

\theoremstyle{remark}
\newtheorem{remark}[theorem]{Remark}

\usepackage{float}
\restylefloat{table}
\restylefloat{figures}
\usepackage{multirow}
\usepackage{graphicx}
\usepackage{tabularx}
\usepackage{color}
\usepackage{here}
\usepackage{algpseudocode,algorithm}

\definecolor{mygray}{gray}{0.6}
\definecolor{lg}{gray}{0.88}
\definecolor{pakistan}{rgb}{0.0, 0.5, 0.0}
\definecolor{kimidori}{rgb}{0.85,0.93,0.3}
\definecolor{mypink}{rgb}{0.9, 0.0, 0.4}
\definecolor{yamabuki}{rgb}{1.0, 0.86, 0.0}
\definecolor{navy}{rgb}{0.0,0.0,0.7}
\definecolor{darkred}{rgb}{0.7,0.0,0.0}

\usepackage{tcolorbox}
\tcbuselibrary{raster,skins} 
\tcbuselibrary{breakable} 

\title[A structure theorem for rooted binary phylogenetic networks] 
{A structure theorem for rooted binary phylogenetic networks and its implications for tree-based networks}
\author{Momoko Hayamizu}
\address{Department of Statistical Modeling\\ The Institute of Statistical Mathematics, Tokyo, Japan}
\email{hayamizu@waseda.jp}
\curraddr[]{Department of Applied Mathematics\\ Waseda University\\ Tokyo\\ Japan} 

\begin{document}
\begin{abstract}
Attempting to recognize a tree inside a phylogenetic network is a fundamental undertaking in evolutionary analysis. In the last few years, therefore, ``tree-based'' phylogenetic networks, which are defined by a spanning tree called a ``subdivision tree'' that is an embedding of a phylogenetic tree on the same leaf-set, have attracted much attention of  theoretical  biologists. However, the application of such networks is still not easy, due to many important computational problems whose time complexities are unknown or not clearly understood.	 
	In this paper, we provide a general framework for solving those various old or new problems on tree-based phylogenetic networks from a coherent perspective, rather than analyzing the complexity of each individual problem or developing an algorithm one by one. More precisely, we establish a structure theorem that gives a way to canonically decompose any rooted binary phylogenetic network $N$ into maximal zig-zag trails that are uniquely determined by $N$, and furthermore use it to characterize the set of subdivision trees of  $N$ in the form of a direct product, in a way reminiscent of the structure theorem for finitely generated Abelian groups. From these main results, we derive a series of linear time (and linear time delay) algorithms for solving the following problems: given a rooted binary phylogenetic network $N$,  1) determine whether or not $N$ has a subdivision tree and find one if there exists any (decision/search problems); 2) measure the deviation of $N$ from being tree-based (deviation quantification problem); 3) compute the number of subdivision trees of $N$ (counting problem); 4) list all subdivision trees of $N$ (enumeration problem); and 5) find a subdivision tree to maximize or minimize a prescribed objective function (optimization problem). All algorithms proposed here are optimal in terms of time complexity. Our results do not only imply and unify various known results in the relevant literature, but also answer many open questions and moreover enable novel applications, such as the estimation of a maximum likelihood tree underlying a tree-based network.   The results and algorithms in this paper still hold true for a special class of rooted non-binary phylogenetic networks. 
\end{abstract}

\maketitle

\section{Introduction}\label{sec:intro}
Phylogenetic networks are widely used to describe reticulate evolution or to represent conflicts in data or uncertainty in evolutionary histories   (\textit{e.g.},~\cite{bryant2004neighbor, SplitsTree4, Huson, MikeBook2}),  but phylogenetic trees are still regarded as a fundamental model of evolution for their ultimate simplicity. In fact, there are numerous situations in which it makes biological sense to try to recognize phylogenetic trees in a phylogenetic tree network. However, given the fact that there are many computationally intractable problems around this subject (\textit{e.g.}, \cite{Linz, van2010locating}), it would be natural to explore special but reasonably wide subclasses of phylogenetic networks with the property that an underlying tree can be easily retrieved.

The above thoughts are closely related to the motivation behind the concepts of ``tree-based phylogenetic networks'' and their underlying ``subdivision trees'' (\textit{i.e.},  phylogenetic trees embedded in tree-based phylogenetic networks in a certain form), which were originally introduced by Francis and Steel in \cite{FS} and have been extensively studied in the last few years (\textit{e.g.}, \cite{Owen, BS2018, FF2020, unrootedTBN2018,  FM2018, newchara2018, UTBN, mathbio2018, JettenOLD, Jetten, pons, Semple2016,  LX}). 
Tree-based phylogenetic networks are a biologically meaningful extension of phylogenetic trees, and form a fairly large class of networks that encompasses many popular subclasses of phylogenetic networks, such as tree-child networks, tree-sibling networks, and reticulation-visible networks \cite{FS}. In \cite{FS}, it was shown that there exists a polynomial time algorithm for finding a subdivision tree of a tree-based phylogenetic network, which led to the expectation that tree-based networks might have some other mathematically tractable properties.

Although the results and questions in \cite{FS}  have prompted discussion on  some computational problems pertaining to tree-based phylogenetic networks and subdivision trees (\textit{e.g.}, \cite{Owen, unrootedTBN2018, newchara2018, JettenOLD, Jetten, pons, LX}), we must emphasize that our present work is more ambitious than previous studies as our goal here is to build a general framework for solving many different problems from a coherent perspective, rather than analyzing the complexity of each problem separately or developing a fast algorithm one by one. The intuition behind our approach is simple:  if we wish to understand a complicated object, it is generally useful to decompose it into smaller, simpler, and more tractable substructures. As is well known, this philosophy is common to the various ``structure theorems'' that have been established in different branches of mathematics, such as the structure theorem for finitely generated Abelian groups (also known as the fundamental theorem of finite Abelian groups) that states that every finitely generated Abelian group can be uniquely decomposed as a direct product of finitely many cyclic groups, in much the same way as the prime factorization of natural numbers. 

In this paper, we establish a ``structure theorem for rooted binary phylogenetic networks'' (Thoerem~\ref{uniquely.decomposable}), which provides a way to canonically decompose any rooted binary phylogenetic network into its intrinsic substructures that exist uniquely. The structure theorem has considerable implications for research on tree-based networks because it does not only provide a new unified perspective to prove some known results in the relevant literature 
through decomposition-based characterizations of tree-based phylogenetic networks
 (Lemma~\ref{iff} and Corollary~\ref{cor:characterization.tbn}), but also, even more importantly, yields a characterization of the set of  subdivision trees of a tree-based phylogenetic network  in the form of a direct product (Theorem~\ref{structure}) in the spirit of the above-mentioned classical structure theorem. Our structural results furnish a series of linear time (and linear time delay) algorithms for a variety of old or new important problems on tree-based networks and subdivision trees. The problems to be discussed in this paper are listed and outlined below, where we start with the simplest problem and then proceed to more advanced ones.
 
 The first is a so-called \textbf{decision/search problems} (Problem~\ref{prob:decision}), which is the most basic object of study in computational complexity theory. As tree-based networks are defined to be phylogenetic networks containing at least one subdivision tree, the problem is formulated as follows: given a rooted binary phylogenetic network $N$, determine whether or not $N$ is tree-based and find a subdivision tree of $N$ if there exists any. This problem has been well studied in the field of combinatorial phylogenetics. In~\cite{FS}, Francis and Steel proved that the decision/search problems can be formulated as the 2-satisfiability problem (2-SAT) and so can be solved in linear time, thus providing an algorithmic characterization of tree-based phylogenetic networks. In \cite{LX}, Zhang gave a different characterization by focusing on matchings in a bipartite graph associated with $N$ and described a simple linear algorithm for the decision part. Independently from \cite{LX}, Jetten \cite{JettenOLD} and Jetten and van Iersel \cite{Jetten} also obtained the same graph-theoretical characterizations in a slightly different form. More recently, Francis~\textit{et al.}~\cite{newchara2018} obtained several new characterizations of tree-based phylogenetic networks, including those in the spirit of Dilworth’s theorem and in terms of matchings in a bipartite graph. In this paper, we provide a new perspective for a unified understanding of these known results and gives a linear time algorithm for solving the decision/search problems in a coherent and straightforward manner.
 
 Related to the decision problem, discussion has also been made on how to compute the deviation measure $\delta(N)$ of a phylogenetic network $N$ from being tree-based \cite{FF2020, newchara2018, JettenOLD, mooiman2018, pons}, which we call the \textbf{deviation quantification problem} (Problem~\ref{prob:deviation}). More precisely, based on the point that any non-tree-based rooted binary phylogenetic network $N$ can be converted into a tree-based network by introducing new leaves \cite{FS}, several studies suggested measuring the degree of the deviation of $N$ from being tree-based by the minimum number $\delta(N)$ of leaves that need to be attached to make $N$ tree-based \cite{newchara2018, JettenOLD} (or by some alternative indices equivalent to $\delta(N)$ \cite{newchara2018}).  The problem of calculating $\delta(N)$ can be viewed as a generalization of the previous decision problem in the sense that we have $\delta(N)= 0$  if and only if $N$ is tree-based. In \cite{newchara2018}, Francis \textit{et al.}~showed that this problem can be solved in $O(n^{3/2})$ time by using a classical algorithm for finding a maximum-sized matching in a bipartite graph associated with $N$, where $n$ denotes the number of vertices of $N$. 
In this paper, we capture the problem in our decomposition-based framework and give a simple formula for $\delta(N)$ (Corollary~\ref{thm:deviation.equals.W-fences}) and an $O(n)$ time algorithm for computing $\delta(N)$, thus improving the current best known bound on the time complexity of the problem.

Another well-studied topic is the following \textbf{counting problem} (Problem~\ref{prob:count}): given a rooted binary phylogenetic network $N$, compute the number $\alpha(N) \geq 0$ of subdivision trees of $N$.  Initially in \cite{FS}, it was noted that this problem might be hard in view of the fact that computing the number of satisfying solutions of 2-SAT is \#P-complete~\cite{valiant1979complexity}. Although this seemed to be a possibility considering a similar tree-counting problem is \#P-complete~\cite{Linz}, several studies independently obtained the formulae for $\alpha(N)$ \cite{portobello2018, JettenOLD, pons} and proposed polynomial time algorithms for counting $\alpha(N)$  \cite{portobello2018, pons}. From an application viewpoint, if one constructs a tree-based phylogenetic network $N$ in some way from biological data, the number $\alpha(N)$ of its subdivision trees can be interpreted as reflecting a certain kind of complexity of $N$ \cite{portobello2018, pons}, as networks having many spanning trees tend to be more complex than those with only a few. In this paper, we derive a simple formula for $\alpha(N)$ that fully elucidates what factors constitute the number $\alpha(N)$ and yields a linear time algorithm for the counting problem.

Despite being closely related to the above counting problem, almost nothing has been done to uncover the complexity of the \textbf{enumeration (listing) problem} (Problem~\ref{prob:listing}): given a tree-based phylogenetic network $N$, list all subdivision trees $T_1,\dots,T_{\alpha(N)}$ of $N$. As mentioned in \cite{FS}, its complexity should be exponential in the number of vertices of $N$ (and it is still exponential in the number of arcs of $N$ because $N$ has $2n+2r-1$ vertices and $2n+3r-2$ arcs, where $n$ and $r$ denote the numbers of leaves and reticulation vertices of $N$, respectively) because $N$ can have exponentially many subdivision trees. We note that, however, this does not deny the existence of efficient listing algorithms. Indeed, in the usual context of algorithm theory, the complexity of enumeration is evaluated in terms of \emph{both} the input and output sizes, not solely the size of the input. In this paper, therefore, we perform a full complexity analysis and provide a linear time delay algorithm (Definition~\ref{dfn:polynomial-delay}), which belongs to the most efficient class of enumeration algorithms, for listing all subdivision trees. As our algorithm can also list a specified number of subdivision trees rather than all trees, it allows for new applications of tree-based phylogenetic networks such as generating subdivision trees uniformly at random.

The last question is the complexity of the following \textbf{optimization problem} (Problem~\ref{prob:optimization}), which has not been treated in the previous literature but is newly defined in this paper: given a tree-based phylogenetic network $N$ in which each arc $a$ is assigned a non-negative weight $w(a)$, find a subdivision tree $T$ of $N$ to maximize (or minimize) the value of a prescribed objective function $f(T)$.
From a statistical standpoint, this problem can be interpreted as modeling the situation where we are given a data-derived phylogenetic network $N$ together with the probability $w(a)$ of each arc $a$ of $N$ and aim to estimate the best tree $T$ underlying $N$ to maximize the likelihood or log-likelihood $f(T)$ of $T$. Despite its fundamental importance, this problem has not been analyzed or even mentioned in the previous literature, because, as we have seen, existing studies only considered tree-based phylogenetic networks in unweighted settings. Even though there are no known results for this  problem, it clearly takes exponential time to calculate the values of the objective function $f$ for all (possibly exponentially many) subdivision trees and to choose the largest or smallest one. Then, the question is whether or not one can find an optimal subdivision tree without doing such an exhaustive search.  Remarkably,  we provide a linear time algorithm for solving the above optimization problem. The key insight behind the method is that our structure theorem enables us to uniquely decompose the objective function $f$ into the sum or product of local objective functions $f_1,\dots,f_\ell$, and by piecing together an optimal solution for each $f_i$, we can automatically obtain a global optimum. Our results on the optimization problem are expected to open up new avenues for statistical applications of tree-based phylogenetic networks, such as the computation of maximum likelihood subdivision trees described above.

The remainder of this paper is organized as follows. In Section~\ref{sec:preliminaries}, we set up basic definitions and notation. Section~\ref{sec:prob} is divided into five subsections to give a more formal description of the above-mentioned five problems along with the relevant results. We also review the basics of analyzing the complexity of the enumeration problem. 
In Section~\ref{sec:results}, we prove the structure theorem for rooted binary phylogenetic networks  (Theorem~\ref{uniquely.decomposable}) and a characterization of the set of subdivision trees of a tree-based phylogenetic network (Theorem~\ref{structure}). In addition to these main results, we also describe some byproducts of Theorem~\ref{uniquely.decomposable} in Subsection~\ref{subsec:characterization.tbn.known.results}: decomposition-based characterizations of tree based phylogenetic networks (Lemma~\ref{iff} and Corollary~\ref{cor:characterization.tbn}) and new proofs of several known results.
In Section~\ref{sec:implications}, we derive from the above results a series of linear time (and linear time delay) algorithms for solving the five problems. We provide pseudocode of each algorithm and a demonstration using a numerical example where appropriate. In Section~\ref{sec:non-binary}, we mention that all results in this paper hold true for a special class of non-binary phylogenetic networks. 
Finally, we conclude the paper by suggesting some open problems and possible directions for further research in Section~\ref{sec:futurework}.

\section{Preliminaries}\label{sec:preliminaries}
Throughout this paper, $X$ denotes a non-empty finite set, and the terms ``graph'' and ``network'' all refer to  finite, simple, acyclic digraphs (directed graphs), which we now define.  A \emph{digraph}  is an ordered pair $(V,A)$ of a set $V$ of vertices and a set $A$ of \emph{arcs} (\textit{i.e.}, directed edges).   
Given a digraph $G$, we write $V(G)$ and $A(G)$ to represent the sets of vertices and arcs of $G$, respectively. If $V(G)$ and $A(G)$ are finite sets, then 
$G$ is said to be \emph{finite}. 
We use the notation $(u,v)$ for an arc $a$ oriented from a vertex $u$ to a vertex $v$, and also write ${\it tail}(a)$ and ${\it head}(a)$ to mean $u$ and $v$, respectively.  
A digraph $G$ is said to be \emph{simple} if   ${\it head}(a)\neq {\it tail}(a)$ holds for any $a\in A(G)$ and $({\it head}(a), {\it tail}(a))\neq ({\it head}(a^\prime),{\it tail}(a^\prime))$ holds for any $a, a^\prime \in A(G)$ with $a\neq a^\prime$.  A simple digraph $G$ is said to be \emph{acyclic}  if $G$ has no cycle, namely, there is no sequence $(a_0, a_1,\dots,a_{k-1})$ of two or more elements of $A(G)$ such that ${\it head}(a_{i-1})={\it tail}(a_{i})$ holds for each $i\in [1,k]$, with indices taken $\mathrm{mod}$ $ k$.

For graphs $G$ and $H$, $G$ is called a \emph{subgraph} of $H$ if both   $V(G) \subseteq V(H)$ and $A(G) \subseteq A(H)$ hold, in which case we write $G \subseteq H$.  
A subgraph $G\subseteq H$ is said to be \emph{proper} if we have either $V(G)\neq V(H)$ or $A(G)\neq A(H)$.  A subgraph $G\subseteq H$ is said to be \emph{spanning} if  $V(G)=V(H)$ holds. 
Given a graph $G$ and a  subset  $A^\prime\subseteq A(G)$, $A^\prime$ is said to \emph{induce the subgraph $G[A^\prime]:=(V(A^\prime), A^\prime)$ of $G$}, where $V(A^\prime)$ denotes a set of the heads and tails of all arcs in $A^\prime$. 
Besides, given a graph $G$  with $|A(G)| \geq 1$ and a partition  $\{A_1,\dots,A_{\ell}\}$ of $A(G)$, the collection  $\{G[A_1], \dots, G[A_{\ell}]\}$ is called a   \emph{decomposition}  of $G$, where a partition of a set $S$ is defined to be a collection of non-empty disjoint subsets of $S$ whose union is $S$.

For an arc $(u, v)$ of a graph $G$, the \emph{subdivision} of $(u, v)$ refers to the operation of introducing a new vertex $x$ and replacing $(u, v)$ with the new consecutive arcs $(u, x)$ and $(x, v)$. Also, any graph that can be obtained from $G$ by subdividing each arc zero or more times is called a \emph{subdivision} of $G$. 

For  a vertex $v$ of a digraph $N$,  the \emph{in-degree} and \emph{out-degree of $v$ in $N$}, denoted by ${\it deg}^-_N(v)$ and ${\it deg}^+_N(v)$, are defined to be the cardinalities of the sets $\{a\in A(N) \mid {\it head}(a)=v\}$ and $\{a\in A(N) \mid {\it tail}(a)=v\}$, respectively. Given an acyclic  digraph $N$, a vertex  $v\in V(N)$   is  called a \emph{leaf} of $N$ if  $({\it deg}^-_N(v), {\it deg}^+_N(v))=(1,0)$ holds. 
 
\begin{definition}\label{dfn:rbpn}
Given a finite set $X$, a \emph{rooted binary phylogenetic $X$-network} is defined to be any finite simple acyclic digraph $N$ that has the following properties:
\begin{enumerate}
  \item there exists a unique vertex $\rho$ of $V(N)$  with ${\it deg}^-_N(\rho)=0$ and ${\it deg}^+_N(\rho)\in \{1,2\}$; 
  \item $X$ is the set of leaves of $N$;
  \item each vertex $v\in V(N)\setminus (X\cup \{\rho\})$ satisfies $\{{\it deg}^-_N(v), {\it deg}^+_N(v)\}= \{1,2\}$.
\end{enumerate}
\end{definition}

In Definition~\ref{dfn:rbpn},  the vertex $\rho$ is called \emph{the root} of $N$, which can be interpreted as the origin of all species that are signified by the leaves of $N$ (\textit{i.e.}, the elements of $X$). In addition, we call  $v\in V(N)\setminus (X\cup \{\rho\})$ a \emph{tree vertex} of $N$ if $({\it deg}^-_N(v), {\it deg}^+_N(v))= (1,2)$ holds, and  a \emph{reticulation vertex} of $N$  otherwise.  In the case when $N$ has no reticulation vertex, $N$ is called a rooted binary phylogenetic $X$-\emph{tree}. 

\begin{definition}[\cite{FS}; see also~\cite{MikeBook2}]  \label{dfn:subdiv}
 If a rooted binary phylogenetic $X$-network $N$ has a spanning tree $T$ of $N$ that is a subdivision of a rooted binary phylogenetic $X$-tree $\widetilde{T}$, then $T$ is called a \emph{subdivision tree of $N$} (and $\widetilde{T}$ is called a \emph{base tree of $N$}). 
\end{definition}

Although subdivision trees are known by different names, such as support trees \cite{FS}, embedded support trees \cite{pons} and embedded spanning trees \cite{FM2018}, we use the terminology from \cite{MikeBook2} in this paper.

\begin{definition}[\cite{FS}]\label{dfn:tbn}
A rooted binary phylogenetic $X$-network $N$ is called a \emph{tree-based phylogenetic network} (\emph{on $X$}) if $N$ has at least one subdivision tree.  
\end{definition}

We note that, in \cite{FS},  tree-based phylogenetic networks were defined in a more algorithmic way using the operation of placing additional vertex-disjoint arcs between arcs of a rooted binary phylogenetic $X$-tree. This original definition helps understand that tree-based networks are such a natural extension of rooted binary phylogenetic $X$-trees that can be interpreted as ``merely trees with additional arcs'', but we here use Definition~\ref{dfn:tbn}, which was shown in \cite{FS} to be equivalent to the original one, like many other studies (\textit{e.g.}, \cite{fischer2018non, newchara2018, mathbio2018}). One advantage of Definition~\ref{dfn:tbn} is that it can be applied to non-binary phylogenetic networks as well. Although such a generalization is not the main focus of this paper, we treat some non-binary phylogenetic networks in Section~\ref{sec:non-binary}.

As we will see later in Theorem~\ref{thm:bijection}, the concept of ``admissible subsets'' played a central role in \cite{FS}. In this paper, we slightly generalize the original definition in \cite{FS} for the sake of technical convenience and use the following Definition~\ref{dfn:admissible}. 

\begin{definition}[\cite{FS}]\label{dfn:admissible}
Suppose $Z$ is any subgraph of a rooted binary phylogenetic $X$-network. We say that a subset $S$ of $A(Z)$ is \emph{admissible} if $S$ satisfies the following conditions: 
\begin{description}
\item [C0] $S$ contains all $(u,v)\in A(Z)$ with ${\it deg}^-_{N}(v)=1$ or ${\it deg}^+_{N}(u)=1$.
\item [C1] for any  $a_1, a_2\in A(Z)$ with ${\it head}(a_1)={\it head}(a_2)$, exactly one of $\{a_1, a_2\}$ is in $S$. 
\item [C2] for any $a_1, a_2\in A(Z)$ with ${\it tail}(a_1)={\it tail}(a_2)$, at least one of $\{a_1, a_2\}$ is in $S$. 
\end{description}
\end{definition}

  \section{Problems and related work}\label{sec:prob}
  In this section,  we give  a more detailed description of the problems to be solved in this paper that were overviewed in Section~\ref{sec:intro}. As before, we  describe the simplest problem first and then gradually move to more advanced ones.
  
\subsection{Decision/search problems} 
   \begin{problem}[\cite{FS}]\label{prob:decision}
   Given a rooted binary phylogenetic $X$-network $N$, determine whether or not  $N$ is a tree-based phylogenetic network on $X$ and find a subdivision tree of $N$ if there exists any.
 \end{problem}

In~\cite{FS}, Francis and Steel showed that the above decision/search problems can be formulated as the 2-satisfiability problem (2-SAT) and so can be solved in linear time, based on the following theorem that characterizes tree-based phylogenetic networks by the existence of an admissible subset (Definition~\ref{dfn:admissible}).

\begin{theorem}[Theorem 1(a) in \cite{FS}]\label{thm:bijection}
   A rooted binary phylogenetic $X$-network $N$ is tree-based if and only if there exists an admissible subset $S$ of $A(N)$. In this case, $S$ induces a subdivision tree $N[S]$ of $N$. Moreover, there exists a bijection between the families of admissible subsets $S$ of $A(N)$ and  of arc-sets of subdivision trees of $N$. 
\end{theorem}

As Problem~\ref{prob:decision} has been intensively studied, it is known that  tree-based networks can be characterized in many other ways. In \cite{LX}, Zhang gave a different characterization in terms of matchings in a bipartite graph associated with $N$ by using Hall's marriage theorem and developed a linear algorithm for the decision part (Theorem~\ref{thm:zhang}). Independently from \cite{LX}, Jetten \cite{JettenOLD} also obtained some characterizations in a similar approach (Theorem 2.4, Corollary 2.7, and Theorem 2.10 in \cite{JettenOLD}), and Jetten and van Iersel \cite{Jetten} showed that  one of them is is virtually the same as Zhang's, albeit slightly different in their expression (Theorem~\ref{thm:jetten.characterization}).  More recently, Francis~\textit{et al.}~\cite{newchara2018} obtained several new characterizations of tree-based phylogenetic networks, in terms of path partitions and antichains (in the spirit of Dilworth’s theorem), as well as via matchings in a bipartite graph. One of the characterizations in \cite{newchara2018} is restated in Theorem~\ref{thm:characterization.disjoint.paths}. 
We will have a closer look at these known results in Subsection~\ref{subsec:characterization.tbn.known.results}.%
     
\subsection{Deviation quantification problem}
 \begin{problem}[\cite{newchara2018}]\label{prob:deviation}
   Given a rooted binary phylogenetic $X$-network $N$, compute the minimum number $\delta(N)$ of leaves that need to be attached to make $N$ tree-based.   
 \end{problem}
 
Related to the decision problem, discussion has also been made on what we call the \textbf{deviation quantification problem} that requires computing the deviation measure $\delta(N)$ of a phylogenetic network $N$ from being tree based \cite{FF2020, newchara2018, JettenOLD, mooiman2018, pons}. More precisely, based on the point that any non-tree-based rooted binary phylogenetic network $N$ can be converted into a tree-based network by introducing new leaves \cite{FS}, several studies suggested measuring the degree of the deviation of $N$ from being tree-based by the minimum number $\delta(N)$ of leaves that need to be attached to make $N$ tree-based \cite{newchara2018, JettenOLD} (or some alternative indices equivalent to $\delta(N)$ \cite{newchara2018}).

  The problem of calculating $\delta(N)$ can be viewed as a generalization of the decision problem in the sense that we have $\delta(N)= 0$  if and only if $N$ is tree-based. In \cite{newchara2018}, Francis \textit{et al.}~showed that this problem can be solved in $O(n^{3/2})$ time, where $n$ denotes the number of vertices of $N$, by applying a classical algorithm for finding a maximum-sized matching in a bipartite graph associated with $N$. No studies obtained a better running time bound.

\subsection{Counting problem}
\begin{problem}[\cite{FS}]\label{prob:count}
Given a rooted binary phylogenetic $X$-network $N$, count the number $\alpha(N)$ ($\geq 0$) of subdivision trees of $N$.
 \end{problem}
To see what Problem~\ref{prob:count} is exactly asking for, we recall that there exists a one-to-one correspondence between the family of admissible subsets of $A(N)$ and the family of arc-sets of subdivision trees of $N$ (Theorem~\ref{thm:bijection}). It follows that the above number $\alpha(N)$ equals the number of admissible subsets of $A(N)$. This is worth noting because two different admissible subsets of $A(N)$ can result in isomorphic trees, as shown  in Figure~\ref{fig:crossovers}. In other words, Problem~\ref{prob:count}, if we take the network in Figure~\ref{fig:crossovers} as input $N$, we must count its isomorphic subdivision trees in duplicate and conclude that $\alpha(N)=4$ holds. We will discuss such counting problems that are different from but related to Problem~\ref{prob:count} in Subsection~\ref{subsec:direction1}.

\begin{figure}[htbp]
\centering
\includegraphics[scale=.8]{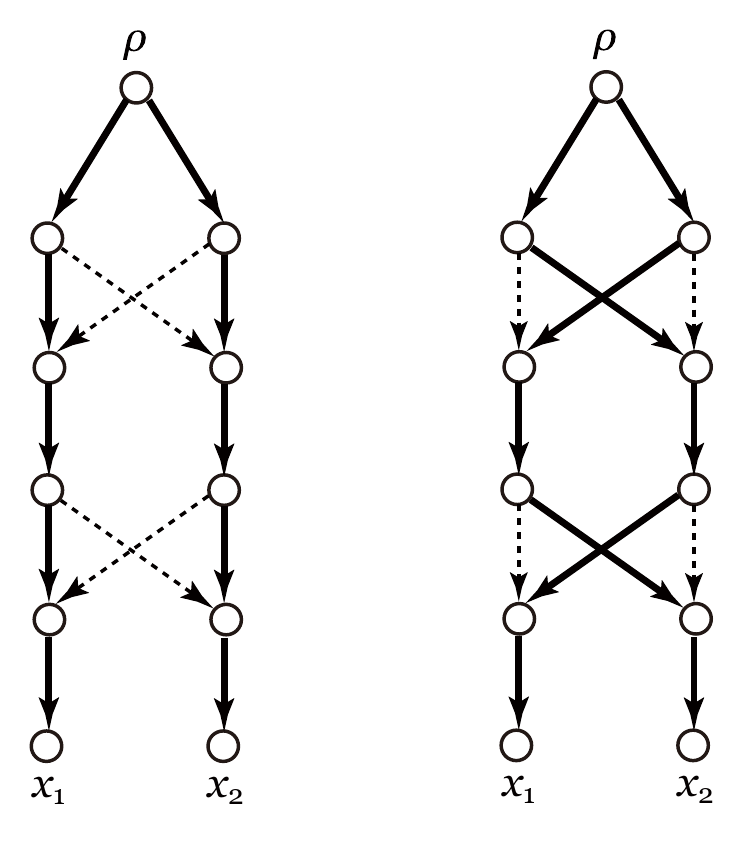}
\caption{An illustration of two distinct subdivision trees that become isomorphic if internal vertices are not distinguished. The arcs of a subdivision tree are shown in bold. 
\label{fig:crossovers}}
\end{figure}

Initially in \cite{FS}, it was noted that this problem might be hard in view of the fact that computing the number of satisfying solutions of 2-SAT is \#P-complete~\cite{valiant1979complexity}, which is a plausible possibility considering a similar tree-counting problem is \#P-complete~\cite{Linz}. However, several independent follow-up studies obtained formulae for $\alpha(N)$ \cite{portobello2018, JettenOLD, pons}  and proposed polynomial time algorithms for counting $\alpha(N)$  \cite{portobello2018, pons}. Nevertheless, a more detailed time complexity analysis has not been provided to date (except for the conference presentation \cite{portobello2018}).

\subsection{Enumeration (listing) problem}
\begin{problem}[\cite{FS}]\label{prob:listing}
Given a tree-based phylogenetic network $N$ on $X$, list all subdivision trees $T_1,\dots, T_{\alpha(N)}$ of $N$.
 \end{problem}
 
 Notice that $\alpha(N)$ in Problem~\ref{prob:listing} is assumed to be positive as the enumeration would not be meaningful otherwise (the same applies to the optimization discussed below). We also note that similarly to counting subdivision trees,  what is required in Problem~\ref{prob:listing} is essentially the list of admissible subsets of $A(N)$.
 Although Francis and Steel \cite{FS} raised the questions on the complexities of both Problems~\ref{prob:count} and \ref{prob:listing}, the complexity of Problem~\ref{prob:listing} remains unknown. For the convenience of the reader, below we recall the basics of analyzing the complexity of enumeration problems.

In general, the number of solutions of   enumeration problems can be exponential in the input size or even infinite, and in such a case, even the most efficient algorithms should take a very long time before generating all solutions. Therefore, the efficiency of enumeration algorithms needs to be defined in a different way than that of algorithms for solving other kinds of computational problems. 
Particularly important in this context is the notion of ``polynomial (time) delay algorithms'' in Definition~\ref{dfn:polynomial-delay}, which was originally introduced by Johnson, Yannakakis and Papadimitriou  in  \cite{johnson1988generating}.  As in Definition~\ref{dfn:polynomial-delay}, this concept incorporates the natural idea that a fast enumerating algorithm should be able to list solutions one after another with a small delay.

\begin{definition}[\cite{johnson1988generating}]\label{dfn:polynomial-delay}
An enumeration algorithm is called a \emph{polynomial (time) delay algorithm} if it can do each of the following  in time bounded by a polynomial function $f$ of the input size $n$:
\begin{itemize}
	\item to output a first solution;
	\item to output a next solution that has not been found yet if there exists any and stop otherwise. 
\end{itemize}
In the case when $f$ is a linear function of $n$, the algorithm is particularly called a \emph{linear (time) delay algorithm}.
\end{definition}

As can be seen easily from Definition~\ref{dfn:polynomial-delay}, polynomial delay algorithms have the remarkable advantage that their running time is linear in the size of the output \cite{goldberg2009efficient}. 
For example, if there exists a polynomial delay algorithm for solving Problem~\ref{prob:listing} such that  each step in Definition~\ref{dfn:polynomial-delay} can be done in  $O(f(n))$ time,  then the algorithm can list $k$ subdivision trees  in  $O(k\cdot f(n))$ time, where $f(n)$ is a  polynomial function of the input size $n$. Note that the converse is not true; in fact, an algorithm being able to generate $k$ solutions in $O(k\cdot f(n))$ time does not necessarily have a polynomial delay, because the time between the output of any two consecutive solutions may not be $O(f(n))$. As indicated by this, polynomial delay algorithms are considered as the most efficient class of enumeration algorithms.

\subsection{Optimization problem}
   \begin{problem}\label{prob:optimization}
Given a tree-based phylogenetic network $N$ on $X$ and its associated weighting function $w:A(N)\rightarrow \mathbb{R}_{\geq 0}$, find a subdivision tree $T$ of $N$ to maximize  the value of the objective function $f(T)=\sum_{a\in A(T)} {w(a)}$.
 \end{problem}

Similarly to Problem~\ref{prob:listing}, $\alpha(N)$ in Problem~\ref{prob:listing} is assumed to be positive again. 
We note that Problem~\ref{prob:optimization} can be converted into a minimization problem by changing the sign, or one could use the objective function $f(T)=\prod_{a\in A(T)}{w(a)}$ by taking the exponential. Typical applications of Problem~\ref{prob:optimization} include the setting where, given a phylogenetic network $N$ together with the probability $w(a)$ of the existence of each arc $a$ of $N$, we wish to estimate a subdivision tree $T$ of $N$ to maximize the likelihood or log-likelihood $f(T)$ of $T$. Despite its direct relevance to such statistical applications, Problem~\ref{prob:optimization} has not been treated in the literature to date. 

An obvious algorithm computes the values $f(T_1),\dots, f(T_{\alpha(N)})$ of the objective function for all subdivision trees of $N$ and then choose the largest (or smallest) one among them. However,  this takes exponential time in the worst case because $N$ can have exponentially many subdivision trees. Then, it would be significant if we could develop a method that can find an optimal solution without doing such an exhaustive search.

\section{Main results}\label{sec:results}
In Subsection~\ref{subsec:decomposition}, we prove the structure theorem that offers a way to canonically decompose a rooted binary phylogenetic network into  ``maximal zig-zag trails'' that exist uniquely for each network (Theorem~\ref{uniquely.decomposable}). 
In Subsection~\ref{subsec:characterization.tbn.known.results}, we touch on decomposition-based characterizations of tree-based phylogenetic networks and new proofs of some known results. In Subsection~\ref{subsec:characterization.set.subdivision.trees}, we establish a characterization of the set of subdivision trees of a tree-based phylogenetic network (Theorem~\ref{structure}).

\subsection{Canonical decomposition of rooted binary phylogenetic networks}\label{subsec:decomposition}
 We note that the idea of ``maximum zigzag trails'' is a simple and natural one, and in fact, a similar concept called ``maximum zig-zag paths'' exists in the literature \cite{JettenOLD, Jetten, LX}. However, those previous studies defined the notion in relation to bipartite undirected graphs and for a different purpose than ours. Therefore, in the following, we give the necessary definitions and terminology which we will use in this paper. 

For a rooted binary phylogenetic $X$-network $N$, we define a  \emph{zig-zag trail in $N$} as a connected subgraph $Z$ of $N$ with $|A(Z)|\geq 1$ such that there exists a permutation $(a_1,\dots,a_m)$ of $A(Z)$ where either ${\it head}(a_i)={\it head}(a_{i+1})$ or ${\it tail}(a_i)={\it tail}(a_{i+1})$ holds for each $i\in [1,m-1]$.  
Any zig-zag trail $Z$ in $N$ can be expressed by an alternating sequence of (not necessarily distinct) vertices and distinct arcs, such as $(v_0, (v_0, v_1), v_1, (v_2, v_1), v_2, (v_2, v_3), \dots,  (v_m, v_{m-1}), v_m)$; however, we will more concisely represent  $Z$ by writing  $v_0>v_1<v_2>v_3< \cdots > v_{m-1} < v_m$ or in reverse order. The notation $(a_1,\dots,a_m)$ may be also used when no confusion arises. 

A zig-zag trail $Z$ in $N$ is said to be \emph{maximal} if $N$ contains no zig-zag trail $Z^\prime$ such that $Z$ is a proper subgraph of $Z^\prime$. It is easy to see that maximal zig-zag trails in $N$ are classified into four types (as mentioned in \cite{LX} as well).  In this paper, borrowing the language of the theory of partially ordered sets,  we  name each of the four as follows 
 (see also Figure~\ref{fig:zigs-labeled}).  
A maximal zig-zag trail $Z$ in $N$ with even $m:=|A(Z)|\geq 4$ is called a \emph{crown} if $Z$  can be written in the form $v_0 < v_1 > v_2 < v_3 > \cdots  > v_{m-2} < v_{m-1} > v_{m} = v_0$; otherwise, it is called a \emph{fence}. 
Furthermore,  a fence $Z$ with odd $m:=|A(Z)|\geq 1$  is called an \emph{N-fence}, which can be expressed as  $Z: v_0 > v_1 < v_2 > v_3 < \cdots >  v_{m-2} < v_{m-1} > v_m$. Also, a fence $Z$ with even $m\geq 2$ is called a \emph{W-fence} if it can be written as $Z: v_0 > v_1 < v_2 > v_3 < \cdots < v_{m-2} > v_{m-1} < v_m$ while it is called an \emph{M-fence} if it can be written as  $Z: v_0 < v_1 > v_2 < v_3 > \cdots > v_{m-2} < v_{m-1} > v_m$. 
For any fence $Z$, its vertices $v_0$ and $v_m$ on both ends are called the \emph{endpoints} of $Z$.

\begin{figure}[htbp]
\centering
\includegraphics[scale=.8]{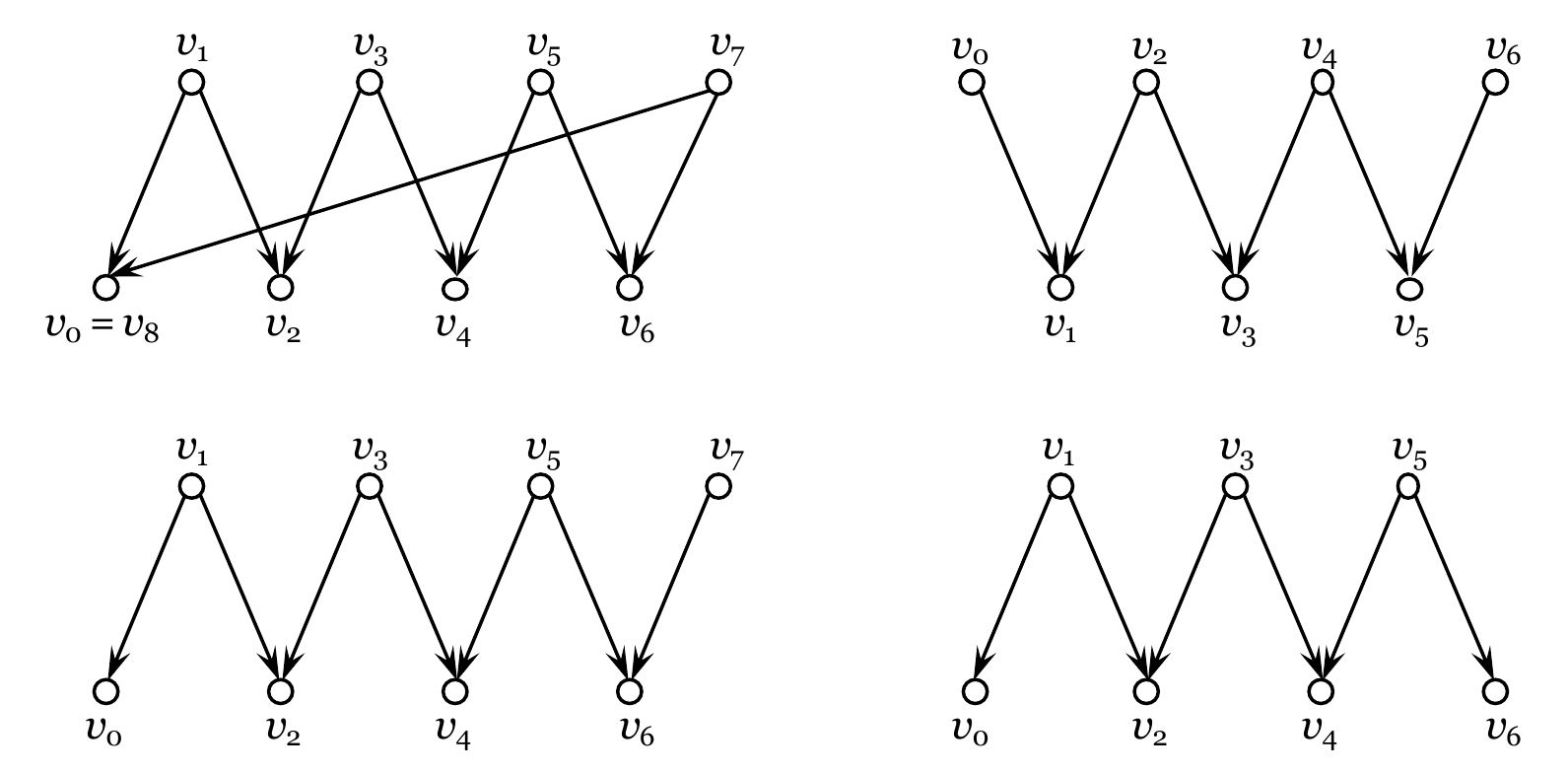}
\caption{An illustration of the four types of maximum zigzag trails of $N$ (modified from Figure~3 in \cite{LX}), each of which we explicitly name in this paper. Top left is a crown. Botton left shows an N-fence. On the right are a W-fence (top) and an M-fence (bottom). Note that this figure only illustrates the simplest cases. For a different  example, see Remark~\ref{rem:atypical} and Figure~\ref{fig:rem-Mfence}.   \label{fig:zigs-labeled}}
\end{figure}

\begin{remark}\label{rem:atypical}
Although the terms fence and crown in the theory of partially ordered sets usually refer to those that can be represented using bipartite graphs as in Figure~\ref{fig:zigs-labeled} (the interested reader is referred to~\cite{Schroder}), in our terminology, such an atypical M-fence as in Figure~\ref{fig:rem-Mfence} is also allowed. 
\end{remark}

\begin{figure}[htbp]
\centering
\includegraphics[scale=.8]
{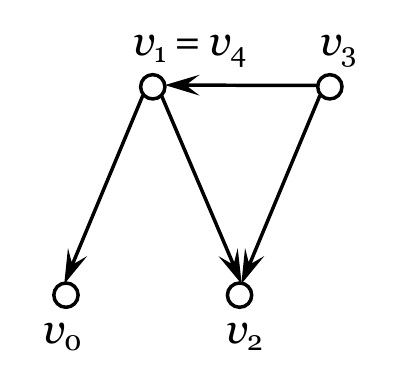}
\caption{An illustrative example mentioned in Remark~\ref{rem:atypical}.  The above  shows a possible maximal M-fence in $N$ that can be written in the form $v_0<v_1>v_2<v_3>v_4$. A similar illustration of an atypical W-fence can also be found in Figure 11(b) in \cite{Jetten}. 
\label{fig:rem-Mfence}}
\end{figure}

We can now state the first main result, which says that every rooted binary phylogenetic $X$-network is composed of the four types of maximal zig-zag trails.

 \begin{theorem}[\textbf{Structure theorem for rooted binary phylogenetic networks}]\label{uniquely.decomposable}
 For any  rooted binary phylogenetic $X$-network $N$,  there exists a unique decomposition $\mathcal{Z}=\{Z_1,\dots,Z_\ell\}$  of $N$ such that each $Z_i\in \mathcal{Z}$ is a maximal zig-zag trail in $N$.
\end{theorem}
\begin{proof} 
The proof is divided into two parts. We first claim that  $A(Z_i)\cap A(Z_j)=\emptyset$ holds for any $i, j \in [1,\ell]$ with $i\neq j$. Suppose, on the contrary to our claim, that $N$ contains two distinct maximal zig-zag trails $Z_i: (a_1,\dots, a_{p-1}, a_p, a_{p+1}, \dots, a_I)$ and $Z_j: (b_1,\dots, b_{q-1}, b_q, b_{q+1}, \dots, b_J)$ such that $a_p=b_q$ holds (\textit{i.e.}, the $p$-th arc $a_p$ of $Z_i$ and the $q$-th arc $b_q$ of $Z_j$ are the same arc of $N$). It follows that $a_{p+1}=b_{q+1}$ holds because if they were different arcs of $N$, then the in-degree of ${\it head}(a_p)$ ($= {\it head}(b_q)$) or out-degree of ${\it tail}(a_p)$ ($= {\it tail}(b_q)$) in $N$ would be greater than two, which contradicts the assumption that $N$ is binary.  By the same reasoning, $a_{p-1}= b_{q-1}$ holds. Repeating this argument yields  $a_1=b_1, \dots, a_p=b_q,\dots, a_I=b_J$ (note that $p=q$ and  $I-p=J-q$ hold by the maximality of $Z_i$ and $Z_j$), but this contradicts the assumption that $Z_i$ and $Z_j$ are distinct.  Thus, the above claim follows. 

Next, we will prove that  for any $a\in A(N)$, there exists a unique element $Z_i$ of $\mathcal{Z}$  with $a\in A(Z_i)$. For any $a\in A(N)$, there exists an obvious  zig-zag trail $Z$ in $N$, that is,  $Z: {\it tail}(a)>{\it head}(a)$. Then, because $N$ is finite, there exists a maximal one $Z_i\in \mathcal{Z}$ with $Z\subseteq Z_i$. By using the first claim, we can conclude that such $Z_i$ is uniquely determined by $a$. 
This completes the proof. 
\end{proof}

\begin{figure}[htbp]
\centering
\includegraphics[scale=.66]{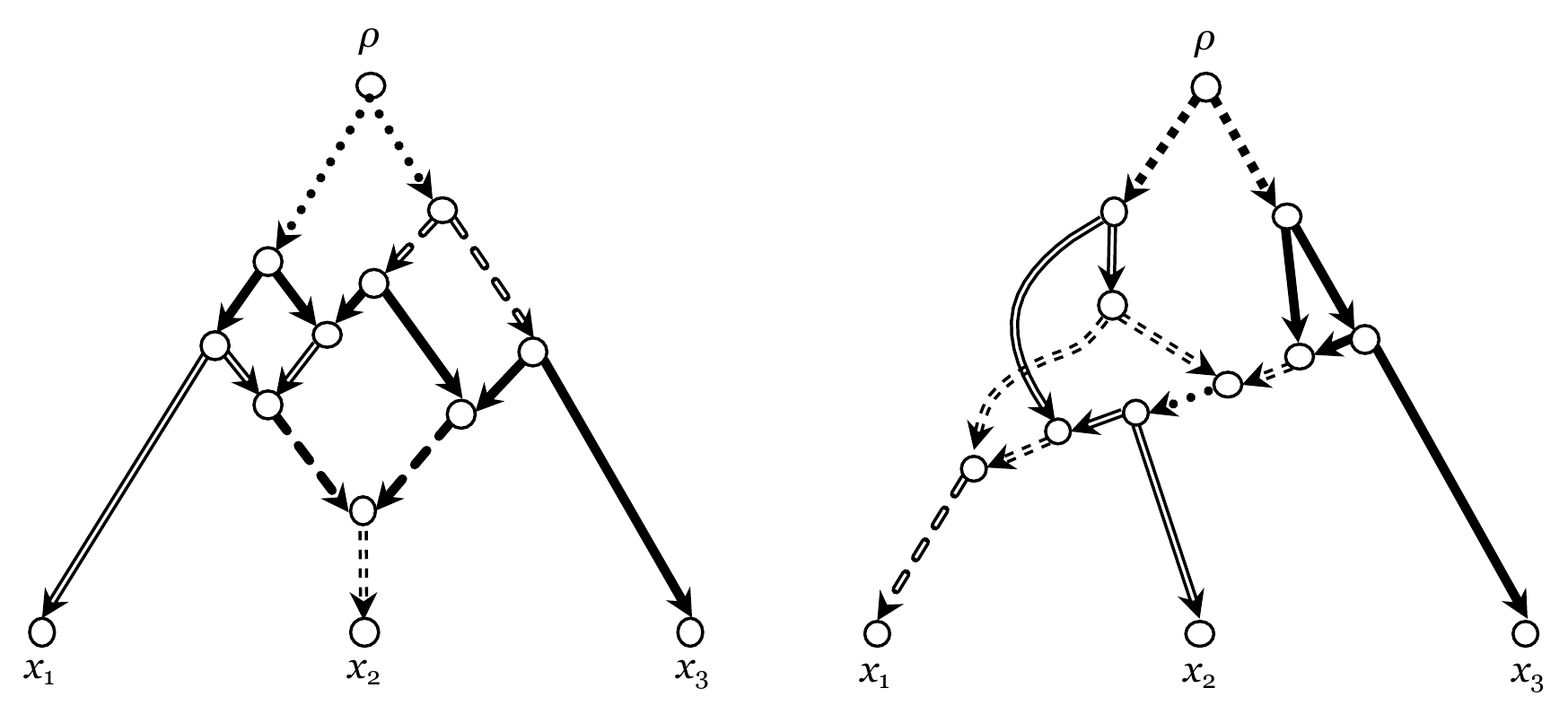}
\caption{Examples of maximal zig-zag trail decomposition as in Theorem~\ref{uniquely.decomposable}, where different types of lines are used to highlight distinct maximal zig-zag trails in each rooted binary $X$-phylogenetic network with $X=\{x_1,x_2,x_3\}$. The  network on the left is decomposed into 3 maximal M-fences, 2 maximal N-fences, and 1 maximal W-fence. On the right is the maximal zig-zag trail decomposition of a network discussed in Figure 3 of \cite{FS}, where the M-fence isomorphic to the one  in Figure~\ref{fig:rem-Mfence} is shown in bold solid arrows. 
\label{fig:demo}}
\end{figure}

\subsection{Decomposition-based characterizations of tree-based phylogenetic networks and new proofs of known results}\label{subsec:characterization.tbn.known.results}
As mentioned before, the problem of characterizing tree based phylogenetic networks has been intensively studied in the past. Although this topic is not a main focus of the present paper, we briefly discuss it for its relevance to Problem~\ref{prob:decision}. As we will explain, our decomposition-based characterizations of tree based phylogenetic networks (Lemma~\ref{iff} and Corollay~\ref{cor:characterization.tbn})  provide alternative proofs of some results in \cite{newchara2018, Jetten, LX}, thus helping the reader relate our work to previous research in the literature (for a more comprehensive list of relevant results, see \cite{newchara2018}).

Let us recall Theorem~\ref{thm:bijection} due to Francis and Steel \cite{FS} that characterizes tree-based phylogenetic networks in terms of the existence of  an admissible arc-set. The following lemma gives a decomposition-based analog of that characterization. 

\begin{lemma}\label{iff}
Let $N$ be a rooted binary phylogenetic $X$-network  
and $\mathcal{Z}=\{Z_1,\dots, Z_\ell\}$ be the maximal zig-zag trail decomposition of $N$. 
Then, $S\subseteq A(N)$ is an  admissible subset of $A(N)$ if and only if $S_i:=  S\cap A(Z_i)$ is an admissible subset of $A(Z_i)$ for each $i\in [1,\ell]$. 
\end{lemma}
\begin{proof}
Our goal is to prove that $S$ satisfies the conditions C0, C1, and C2 in Definition~\ref{dfn:admissible} if and only if  for  each $i\in[1,\ell]$, the following C0$^\prime$, C1$^\prime$, and C2$^\prime$ hold:
\begin{description}
    \item [C0$^\prime$] $S_i$ contains all $(u,v)\in A(Z_i)$ with ${\it deg}^-_{Z_i}(v)=1$ or ${\it deg}^+_{Z_i}(u)=1$; 
    \item [C1$^\prime$] for any $a_1, a_2 \in  A(Z_i)$ with ${\it head}(a_1)={\it head}(a_2)$, exactly one of $\{a_1, a_2\}$ is in $S_i$;
    \item [C2$^\prime$] for any  $a_1, a_2 \in  A(Z_i)$ with ${\it tail}(a_1)={\it tail}(a_2)$, at least one of $\{a_1, a_2\}$ is in $S_i$. 
  \end{description}
If  $(u,v)\in A(N)$ satisfies ${\it deg}^-_N(v)=1$ (or ${\it deg}^+_N(v)=1$), then there exists a unique element $Z_i$ of $\mathcal{Z}$ with $(u, v)\in A(Z_i)$ and ${\it deg}^-_{Z_i}(v)=1$ (or   ${\it deg}^+_{Z_i}(v)=1$) by   Theorem~\ref{uniquely.decomposable}. The converse also holds as  $Z_i$ would not be maximal otherwise.  
Theorem~\ref{uniquely.decomposable} also implies that $S$ is partitioned into $S_1,\dots,S_\ell$ (note that no element of $S$ is empty).  Thus,  we can assert that $S$ satisfies C0 if and only if  C0$^\prime$ holds  for each $i\in[1,\ell]$. 
By similar reasoning,  we can deduce that  $\{a_1, a_2\}\subseteq A(N)$ satisfies  ${\it head}(a_1)={\it head}(a_2)$  if and only if there exists a unique element $Z_i$ of $\mathcal{Z}$ such that $\{a_1, a_2\}\subseteq A(Z_i)$ has the same property. Recalling that $ \{S_1,\dots,S_\ell\}$ is a partition of $S$, we have  $S\cap \{a_1,a_2\} = S_i \cap \{a_1,a_2\}$ for any $i\in [1,\ell]$ and any $a_1,a_2 \in A(Z_i)$ with ${\it head}(a_1)={\it head}(a_2)$. Hence,   $S$ satisfies C1  if and only if  C1$^\prime$ holds for each $i\in[1,\ell]$. The same arguments derive the desired conclusion regarding C2 and  C2$^\prime$. This completes the proof. 
\end{proof}

Notice that Lemma~\ref{iff} already indicates an advantage of the structure theorem. In fact, the lemma greatly simplifies the discussion on subdivision trees of $N$ because it says that what we need to consider is only the admissible arc-sets within each substructure $Z_i$ of $N$, rather than the admissible arc-sets in the entire network $N$.

Lemma~\ref{iff} also allows us to look at some different results in the literature from a unified and coherent viewpoint. Indeed, Lemma~\ref{iff} immediately gives a characterization of tree-based phylogenetic networks  that is equivalent to those that were independently obtained by  Zhang \cite{LX} and  by Jetten \cite{Jetten} in a matching-based approach, as we now explain. For the convenience of the reader, below we restate the relevant definitions and results from \cite{Jetten, LX} (for a fuller summary, see Subsection 3.1 in \cite{pons}). For a rooted binary phylogenetic $X$-network $N$, a reticulation vertex of $N$ is called \emph{type~0} if its parents  are both reticulation vertices and \emph{type~1} if one parent is reticulation vertex and the other is a tree vertex \cite{LX}. Also, a non-leaf vertex of $N$ is called \emph{omnian} if its children are all reticulation vertices \cite{JettenOLD, Jetten}. 

\begin{theorem}[Theorem 1 in \cite{LX}]\label{thm:zhang}
	A rooted binary phylogenetic $X$-network $N$ is tree-based if and only if (i) there is no type 0 reticulation vertex in $N$ and (ii) no two type 1 vertices are connected by a zigzag path.
\end{theorem}

\begin{theorem}[Corollary 2.11 in \cite{Jetten}]\label{thm:jetten.characterization}
	 A rooted binary phylogenetic $X$-network $N$ is tree-based if and only if it contains no zig-zag path $(o_1, r_1, \dots, o_k, r_k, o_{k+1})$, with $k \geq 1$, in which $r_1, \dots, r_k$ are reticulations, $o_1, \dots, o_{k+1}$ are omnians and $o_1$ and $o_{k+1}$ are reticulations as well as omnians.
\end{theorem}

In the language of this paper, the above two results are equivalent to the following. 
\begin{corollary}\label{cor:characterization.tbn}
	Let $N$ be a rooted binary phylogenetic $X$-network  and  $\mathcal{Z}=\{Z_1,\dots,Z_\ell\}$  be the maximal zig-zag trail decomposition of $N$.  Then, $N$ is a tree-based phylogenetic network on $X$ if and only if no element $Z_i\in \mathcal{Z}$ is a W-fence.
\end{corollary}
\begin{proof}
Recalling that an admissible subset of $A(Z_i)$ is defined by the conditions \textbf{C0}$^\prime$, \textbf{C1}$^\prime$, and \textbf{C2}$^\prime$, we can see that there exists an admissible subset of $A(Z_i)$ if and only if $Z_i$ is not a W-fence. Then, by Theorem~\ref{thm:bijection} and Lemma~\ref{iff}, maximal W-fences are the only forbidden substructure for tree-based phylogenetic networks. This completes the proof.
\end{proof}

One may wonder why maximal W-fences are forbidden in tree-based phylogenetic networks.  
As it is instructive to see what happens  if there is a maximal W-fence, we next recall one of the characterizations of tree-based phylogenetic networks in terms of path partitions and antichains that were obtained by Francis~\textit{et~al.}~in \cite{newchara2018}, and give an alternative proof of their result from the perspective of this paper. 

\begin{theorem}[Theorem 2.1 (I), (III) in \cite{newchara2018}]\label{thm:characterization.disjoint.paths}
Let $N$ be a rooted binary phylogenetic $X$-network. Then, the following are equivalent:
\begin{itemize}
	\item $N$ is tree-based;
	\item for any subset $V^\prime$ of $V(N)$, there exists a set $\mathcal{P}$ of vertex disjoint (directed) paths in $N$ each ending at a leaf in $X$ such that each element of $V^\prime$ is on exactly one path in $\mathcal{P}$.
\end{itemize}
\end{theorem}
\begin{proof}
Assume that $N$ is not tree-based. Then, by  Corollary~\ref{cor:characterization.tbn}, $N$ has no maximal W-fence,  which is denoted by $Z: v_0 > v_1 < v_2 > v_3 < \cdots < v_{m-2} > v_{m-1} < v_m$. 
If we partition the set $V(Z)=\{v_0,\dots, v_m\}$ into $V_{\it even}:=\{v_0, v_2, \dots, v_{m-2}, v_m\}$ and  $V_{\it odd}:=\{v_1, v_3, \dots, v_{m-3}, v_{m-1}\}$ (recall that $m$ is even),   
it is obvious that  $V_{\it even}$ is a subset of $V(N)$ and  that $|V_{\it even}|>|V_{\it odd}|$ holds. Then, we claim that $N$ does not have a set $\mathcal{P}$ of vertex-disjoint paths each ending at a different leaf in $X$ such that each element of $V_{\it even}$ is on exactly one path. Suppose, on the contrary, there exists such a set $\mathcal{P}$ of paths. Then, for each path $P$ in $\mathcal{P}$,  the following two statements hold: 1) once $P$ passes through a vertex $v$ in $V_{\it even}$, $P$ must go further through a vertex in $V_{\it odd}$ adjacent to $v$ (because otherwise $P$ cannot go outside of $Z$ and therefore cannot reach a leaf in $X$); 2) $P$ cannot pass through the same element of $V_{\it even}$ twice or more before ending at a leaf in $X$ (because otherwise $P$ would not be a path in $N$). Then, for each $P$ in $\mathcal{P}$, $|V(P)\cap V_{\it even}|\leq |V(P)\cap V_{\it odd}|$ holds. Taking the union over all $P$, we have $|V_{\it even}|\leq |V_{\it odd}|$, which is a contradiction. Hence, the above claim indeed holds, so we can conclude that if $N$ is not tree-based then $N$ does not satisfy the second condition in Theorem~\ref{thm:characterization.disjoint.paths}.	
To prove the converse, assume that $N$ is tree-based. Notice that we do not need to consider every subset $V^\prime$ of $V(N)$ but only need to show that there exists a set $\mathcal{P}$ of vertex disjoint paths in $N$ each ending at a leaf in $X$ such that each element of $V(N)$ is on exactly one path in $\mathcal{P}$. As can be seen easily, $\mathcal{P}$ can be constructed from a subdivision tree $T$ of $N$ by removing one of the arcs outgoing from each branching point of $T$. This completes the proof.
\end{proof}

 \subsection{Characterization of the set of subdivision trees}\label{subsec:characterization.set.subdivision.trees}
From now on, let us consider an ordered set $(Z_1,\dots, Z_\ell)$ of  maximal zig-zag trails in $N$ so that we can identify a subgraph $G$ of $N$ with a direct product  $\prod_{i\in[1,\ell]}{A(G)\cap A(Z_i)}$. In addition,  for each $i\in [1,\ell]$, we represent  the set $A(Z_i)$ using a sequence $(a_1,\dots,a_{m_i})$ of the elements of $A(Z_i)$  that form the zig-zag trail in this order, where $m_i:=|A(Z_i)|$. This allows us to  encode an arbitrary subset of $A(Z_i)$ using a 0-1 sequence of length $m_i$; for example, given $A(Z_i)=(a_1, a_2, a_3, a_4, a_5)$, we can specify the subset $\{a_2, a_4, a_5\}\subseteq A(Z_i)$ by the sequence $\langle 0\, 1\, 0\, 1\, 1\rangle=\langle (01)^2 1 \rangle$. For each $i\in [1,\ell]$, let $\mathcal{S}(Z_i)$ be  a family of subsets of $A(Z_i)$ that can be expressed as follows.   
  
\begin{align}\label{sequence}
\mathcal{S}(Z_i):= 
  \begin{cases}
      \bigl\{ \langle (10)^{m_i/2}\rangle,\ \langle (01)^{m_i/2} \rangle\bigr\} & \text{if $Z_i$ is a crown;} \\
      \bigl\{ \langle 1(01)^{(m_i-1)/2}\rangle \bigr\}  & \text{if $Z_i$ is an N-fence;} \\ 
    \bigl\{ \langle 1(01)^p(10)^q 1 \rangle \mid  p,q\in \mathbb{Z}_{\geq 0}, p+q=(m_i-2)/2\bigr\} & \text{if $Z_i$ is an M-fence.
    } 
   \end{cases}
\end{align}

Note that by virtue of symmetry, the above sequence representation does not depend on the direction in which the arcs of $Z_i$ are ordered.
For example, when $Z_i$ is an N-fence, the sequence $\langle 1(01)^{(m_i-1)/2}\rangle$ and its reverse ordering are identical.

\begin{theorem}[\textbf{Characterization of the set of subdivision trees  of a tree-based phylogenetic network}]\label{structure}
Let $N$ be a tree-based phylogenetic network on $X$  and  $\mathcal{Z}=\{Z_1,\dots,Z_\ell\}$  be the maximal zig-zag trail decomposition of $N$.  Then,  the collection $\mathcal{T}$
 of subdivision trees of $N$ are characterized by 
 \begin{align}\label{eq:subdivision.trees.characterization}
\mathcal{T}=\prod_{i\in[1,\ell]}{\mathcal{S}(Z_i)},
\end{align} 
where $\mathcal{S}(Z_i)$ is defined in the equation~$(\ref{sequence})$. 
\end{theorem}
\begin{proof}
Let $Z$ be an arbitrary element of $\mathcal{Z}$  
with  $A(Z)=(a_1, \dots,  a_m)$. Using the conditions C0$^\prime$, C1$^\prime$, and C2$^\prime$ in the proof of Lemma~\ref{iff}, we now enumerate all 0-1 sequences $\langle x_1\, \dots\, x_m\rangle$ corresponding to the admissible subsets  of $A(Z)$  in the following  four cases (see also Figure~\ref{fig:enumerate}).   

When $Z$ is a crown $v_0 > v_1 < v_2 > v_3 < \cdots  < v_{m-2} > v_{m-1} < v_{m} = v_0$, 
the condition C0$^\prime$ does not apply. Repeated application of the conditions C1$^\prime$ and C2$^\prime$ derives the only solution $\langle 1\, 0\, 1\, 0\dots 1\, 0\rangle$ from $x_1=1$. Similarly, $x_1=0$ implies $\langle 0\, 1\, 0\, 1\, \dots 0\, 1\rangle$. 
This proves that a family  of all admissible subsets of $A(Z)$ is given by $\mathcal{S}(Z)=\bigl\{ \langle (10)^{m/2}\rangle,\ \langle (01)^{m/2} \rangle\bigr\}$. 

When $Z$ is an N-fence $v_0 > v_1 < v_2 > v_3 < \cdots >  v_{m-2} < v_{m-1} > v_m$, 
exactly one of its endpoints $v_0$ is a reticulation vertex of $N$. 
Then, the condition C0$^\prime$ gives  $x_1=1$, which implies $\langle 1\, 0\, 1\, 0\, \dots 0\, 1\rangle$ as in the previous case. 
This proves that $\langle 1(01)^{(m-1)/2}\rangle$ is the only  admissible subset of $A(Z)$.

When $Z$ is an M-fence $v_0 < v_1 > v_2 < v_3 > \cdots > v_{m-2} < v_{m-1} > v_m$, we have $x_1=1$ and $x_m=1$ again but the other values   $x_2,\,  \dots,\,  x_{m-2}$ are left undetermined. 
We claim that a family of all admissible subsets of $A(Z)$ with $|A(Z)|=m\geq 2$ is given by $\mathcal{S}(Z)=\bigl\{ \langle 1(01)^p(10)^q 1 \rangle \mid  p,q\in \mathbb{Z}_{\geq 0}, p+q=(m-2)/2\bigr\}$. 
The proof is by induction on the length $m$ (recall that $m$ is even).  
The assertion is trivial for $m=2$. We consider  the  two cases according to the value of $x_{m+1}$ in the sequence $\langle 1\ x_2\ \dots\ x_{m+1}\ 1\rangle$ of length $m+2$. 
When $x_{m+1}=1$ holds,  $\langle 1 (01)^{m/2}\, 1\rangle = \langle 1 (01)^{m/2} (10)^0\, 1\rangle$  is the only admissible subset  of $A(Z)$ having this form. 
When $x_{m+1}=0$ holds, this only implies  $\langle 1\, x_2\, \dots\ x_{m-1}\, 1\, 0\, 1\rangle$. By the induction hypothesis,  the family of admissible subsets having this form  consists of the sequences $\langle 1(01)^p (10)^q 1\rangle$ with $p\in  \mathbb{Z}_{\geq 0}$, $q\in \mathbb{N}$, and $p+q=m/2$. This proves the claim.

Thus, whenever $\mathcal{S}(Z)$ is defined, $\mathcal{S}(Z)$ is identical to the family of admissible subsets of $A(Z)$.
This completes the proof.
\end{proof}

\begin{figure}[htbp]
\centering
\includegraphics[scale=.7]{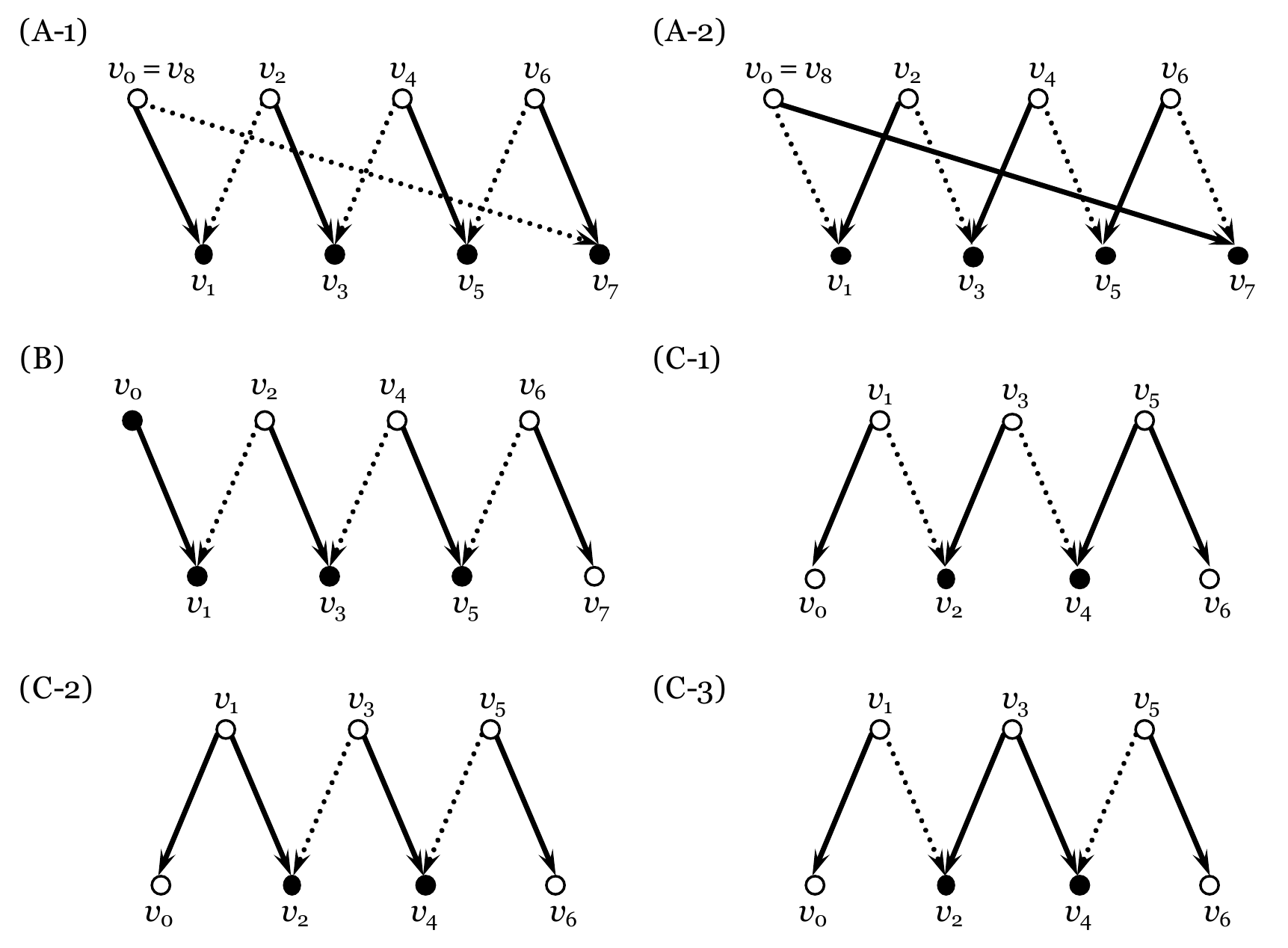}
\caption{Illustrative examples for the proof of Theorem~\ref{structure}.  
The arcs of $Z$ labeled ``0'' are shown in dotted lines   while solid lines indicate those labeled ``1''. 
The vertices of $Z$ are colored in black  if they are reticulation vertices of $N$ and  in white otherwise. (A-1) and (A-2) describe the two  admissible arc-sets  of a crown with 8 arcs that are expressed as $\langle (10)^{4}\rangle$ and $\langle (01)^{4} \rangle$, respectively. (B) shows the only admissible arc-set  of a maximal N-fence with 7 arcs, $\langle 1(01)^{3} \rangle$. The other three are the   solutions for a maximal M-fence with 6 arcs, that is,
 $\langle 1(01)^2 1 \rangle$,  $\langle 1(10)^2 1 \rangle$,  $\langle 1(01)^1 (10)^1 1 \rangle$. 
\label{fig:enumerate}}
\end{figure}

\section{Algorithmic implications of the main results}\label{sec:implications}
As shown in this section, Theorem~\ref{structure} furnishes a series of linear time (and linear time delay) algorithms for the problems described in Section~\ref{sec:prob}. The algorithms given here all start with the pre-processing of decomposing the input network into maximal zig-zag trails (Algorithm~\ref{algm}). 

\begin{algorithm}
\caption{\textsc{Pre-processing} (Decompose a rooted binary phylogenetic $X$-network $N$ into $\ell\geq 1$ maximal zig-zag trails)\label{algm}}
\begin{algorithmic}
\Require{A rooted binary phylogenetic $X$-network $N$}
\Ensure{The maximal zig-zag trail decomposition $\mathcal{Z}$ of $N$}
\State{initialize $\mathcal{Z}:=\emptyset$,   $G:=N$}
\ForAll {arc $a\in A(G)$}
\State{compute a (unique) maximal zig-zag trail $Z$ in $G$ that contains arc $a$}
\State{update $\mathcal{Z}:=\mathcal{Z}\cup \{Z\}$,    $G:=G[A(G)\setminus A(Z)]$}
\EndFor
\State output $\mathcal{Z}$ and \textbf{halt} 
\end{algorithmic}
\end{algorithm}

\begin{proposition}\label{prop:preprocessing.running.time}
 For any  rooted binary phylogenetic $X$-network $N$, Algorithm~\ref{algm} can compute the maximal zig-zag trail decomposition  $\mathcal{Z}=\{Z_1,\dots,Z_\ell\}$  of $N$  in $O(|A(N)|)$ time. 
\end{proposition}
\begin{proof}
In Algorithm~\ref{algm}, each arc of $N$ is visited exactly once. Therefore, the above decomposition $\mathcal{Z}$ of $N$ can be obtained in $O(|A(N)|)$ time. This completes the proof.
\end{proof}

\subsection{Linear time algorithm for the decision/search problems} 
Corollary~\ref{cor:characterization.tbn} gives the obvious algorithm that determines whether or not $N$ is tree-based by checking whether or not $N$ has a maximal W-fence. Also, Theorem~\ref{structure} yields the simple algorithm that selects an arbitrary element of the set $\mathcal{T}$ to retrieve a subdivision tree of $N$. Putting them together, we get an algorithm for solving  Problem~\ref{prob:decision}  described in Algorithm~\ref{algm:decision.search}.

\begin{corollary}
\label{thm:decision.search}
	Algorithm~\ref{algm:decision.search} can solve Problem~\ref{prob:decision} in $O(|A(N)|)$ time.
\end{corollary}
\begin{proof}
By Proposition~\ref{prop:preprocessing.running.time},  the maximal zig-zag trail decomposition $\mathcal{Z}$ of $N$ can be obtained in $O(|A(N)|)$ time. For each $Z_i\in \mathcal{Z}$, one can determine in $O(|A(Z_i)|)$ time whether $Z_i$ is a crown, M-fence, N-fence, or W-fence. In the case when $Z_i$ is not a W-fence, selecting an arbitrary sequence $\sigma_i$ in $\mathcal{S}(Z_i)$ and converting $\sigma_i$ into $A_i$ requires $O(1)$ time and $O(|A(Z_i)|)$ time, respectively. It takes $O(A(N))$ time to construct $T$. Hence, Algorithm~\ref{algm:decision.search} can solve Problem~\ref{prob:decision} in $O(|A(N)|)$ time. 
\end{proof}

\begin{remark}\label{rem:theta.running.time}
Algorithm~\ref{algm:decision.search} is optimal in terms of time complexity (\textit{i.e.}, it achieves the best possible running time) because it requires $\Omega(|A(N)|)$ time just to read $N$ as an input. Thus, Corollary~\ref{thm:decision.search} implies that Problem~\ref{prob:decision} can be solved in $\Theta(|A(N)|)$ time. Similar arguments apply to the other algorithms in this paper.
\end{remark}

\begin{algorithm}
\caption{\textsc{Decision/search}  (Problem~\ref{prob:decision}) \label{algm:decision.search}}
\begin{algorithmic}
\Require{A rooted binary phylogenetic $X$-network $N$}
\Ensure{A subdivision tree $T$ of $N$ if $N$ is tree-based and ``No'' otherwise}
\State{call \textsc{Pre-processing} (Algorithm~\ref{algm}) to obtain $\mathcal{Z}=\{Z_1,\dots, Z_\ell\}$}
\State initialize $\sigma_1, \dots, \sigma_\ell, A_1,\dots, A_\ell := \emptyset$
\ForAll {maximal zig-zag trail $Z_i\in \mathcal{Z}$}
\If{$Z_i$ is a W-fence}
\State output ``No'' and \textbf{halt} 
\Comment Corollary~\ref{cor:characterization.tbn}
\Else
\State{$\sigma_i:=$ an arbitrary sequence in $\mathcal{S}(Z_i)$ in the equation~(\ref{sequence})}
\State $A_i:=$ the set of arcs of $N$ that are specified by  the 1s in $\sigma_i$ 
\EndIf
\EndFor
\State $T:=(V(N), A_1\cup \dots \cup A_\ell)$
\Comment Theorem~\ref{structure}
\State output $T$ and \textbf{halt}
\end{algorithmic}
\end{algorithm}

\subsection{Linear time algorithm for the deviation quantification problem}\label{subsec:deviation.quantification}
Even though a rooted binary phylogenetic $X$-network $N$ is not tree-based if there is a W-fence in the maximal zig-zag trail decomposition $\mathcal{Z}$ of $N$, if we create a rooted binary phylogenetic $(X\cup \{y\})$-network $N^\prime$ by  attaching a new leaf $y$ to one of the arcs of the W-fence, then the W-fence in $N$ becomes two N-fences in $N^\prime$. Thus, any rooted binary phylogenetic $X$-network can be made tree-based by introducing the same number of new leaves as the number of W-fences (the same conclusion was obtained in Theorem~2.11 in \cite{JettenOLD} by a different argument). Moreover,   because  each additional leaf can only  break one W-fence, it is also necessary to attach the above number of new leaves in order to eliminate all W-fences. Hence, we obtain the following.

\begin{corollary}\label{thm:deviation.equals.W-fences}
Let $N$ be a rooted binary phylogenetic $X$-network, $\delta(N)\geq 0$ be the minimum number of leaves that need to be attached to make $N$ a tree-based phylogenetic network on $X$, and $\Delta$ be the number of W-fences in the maximal zig-zag trail decomposition $\mathcal{Z}$ of $N$. Then, we have $\delta(N)=\Delta$. 
\end{corollary}

Corollary~\ref{thm:deviation.equals.W-fences} yields a simple algorithm (Algorithm~\ref{algm:deviation.quantification}) that solves  Problem~\ref{prob:deviation} by counting the number $\Delta$ of maximal W-fences of $N$. As the running time of the algorithm is clearly $O(|A(N)|)$ by Proposition~\ref{prop:preprocessing.running.time}, we also obtain   Corollary~\ref{cor:deviation.quantification.complexity}. 

\begin{corollary}\label{cor:deviation.quantification.complexity}
	Algorithm~\ref{algm:deviation.quantification} can solve Problem~\ref{prob:deviation} in $O(|A(N)|)$ time.
\end{corollary}

As in Remark~\ref{rem:theta.running.time}, the running time of Algorithm~\ref{algm:deviation.quantification} is the best possible and Problem~\ref{prob:deviation} can be solved in $\Theta(|A(N)|)$ time.

\begin{algorithm}
\caption{\textsc{Deviation quantification} (Problem~\ref{prob:deviation})\label{algm:deviation.quantification}}
\begin{algorithmic}
\Require{A rooted binary phylogenetic $X$-network $N$}
\Ensure{The minimum number $\delta(N)\geq 0$ of leaves that need to be attached to make $N$ tree-based}
\State{call \textsc{Pre-processing} (Algorithm~\ref{algm}) to obtain $\mathcal{Z}=\{Z_1,\dots, Z_\ell\}$}
\State initialize $\Delta:=0$
\ForAll {maximal zig-zag trail $Z_i\in \mathcal{Z}$}
\If{$Z_i$ is a W-fence}
\State update $\Delta:=\Delta+1$
\EndIf
\EndFor
\State $\delta(N):=\Delta$
\Comment Corollary~\ref{thm:deviation.equals.W-fences}
\State output $\delta(N)$ and \textbf{halt}
\end{algorithmic}
\end{algorithm}

\begin{remark}\label{rem:deviation.previous.method}
	By Corollary~\ref{cor:deviation.quantification.complexity}, we have improved the current best known bound on the time complexity of Problem~\ref{prob:deviation} that was shown in \cite{newchara2018}. In \cite{newchara2018}, Francis~\textit{et~al.} proposed a polynomial time algorithm for solving Problem~\ref{prob:deviation}, but the running time of their  algorithm has to be  $O(\sqrt{|2V(N)|}\cdot|A(N)|)=O(|V(N)|)^{3/2}$  because the algorithm is  derived from the equation $\delta(N)=|V(N)|-|X|-m(G_N)$ (Lemma~4.1 in \cite{newchara2018}), where  $m(G_N)$ denotes the maximum size of matchings in a bipartite graph $G_N$ whose vertex bipartition of $G_N$ is formed by two copies of $V(N)$ and whose edge-set contains an edge between $u \in V_1$ and  $v \in V_2$ precisely if $(u, v)$ is an arc of $N$. Thus, their method  involves the step of finding the maximum-sized matching in $G_N$, which has $2|V(N)|$ vertices and $|A(N)|$ edges and uses the Hopcroft-Karp algorithm \cite{hopcroft1973} for that purpose. In contrast, Algorithm~\ref{algm:deviation.quantification} requires only $O(|A(N)|)$ time because it simply decomposes $N$ into maximal zig-zag trails and checks how many maximal W-fences exist.
\end{remark}

	For the benefit of the interested reader, we mention other deviation indices described in \cite{newchara2018}.  
	Francis~\textit{et al.}~\cite{newchara2018} defined the following two  as well as $\delta(N)$:
	\begin{itemize}
	\item The minimum number $\ell(N)$ of leaves in $V(N) \setminus X$ that must be present as leaves in a rooted spanning tree of $N$
	\item The minimum number $p(N) = d(N) - |X|$, where $d(N)$ is the smallest number of vertex disjoint paths that partition the vertices of $N$
\end{itemize}
In \cite{newchara2018}, it was shown that $\delta(N)=\ell(N)=p(N)$ holds (Theorem 4.3 in \cite{newchara2018}). Then, it automatically follows that  Algorithm~\ref{algm:deviation.quantification} can compute $\ell(N)$ and $p(N)$ as well as $\delta(N)$   in $O(|A(N)|)$ time.

\subsection{Linear time algorithm for the counting problem} 
Combining Corollary~\ref{cor:characterization.tbn} that determines when the number of subdivision tree equals zero and Theorem~\ref{structure} that characterizes the set of subdivision trees of a tree-based phylogenetic network in the form of a direct product, we can immediately obtain the following.

\begin{corollary}\label{cor:formula}
Let $N$ be a rooted binary phylogenetic $X$-network that has $\alpha(N)\in \mathbb{Z}_{\geq 0}$ subdivision trees  and  $\mathcal{Z}=\{Z_1,\dots,Z_\ell\}$  be the maximal zig-zag trail decomposition of $N$.  Then,  
$\alpha(N)=\prod_{i=1}^{\ell}{\alpha(Z_i)}$ holds,  where  
\begin{align}\label{eq:counting}
  \alpha(Z_i) = 
  \begin{cases}
      0 & \text{if $Z_i$ is a W-fence;} \\
      1 & \text{if $Z_i$ is an N-fence;} \\ 
    2 & \text{if $Z_i$ is a crown;} \\
        |A(Z_i)|/2 & \text{if $Z_i$ is an M-fence.} \\
   \end{cases}
\end{align}
\end{corollary}

For the benefit of the reader, let us look at the relevant results in \cite{JettenOLD, pons}  to see how our results helps understand  them. In~\cite{JettenOLD}, Jetten considered the problem of counting the number $\beta(N)$ of ``base trees'' of a tree-based phylogenetic network $N$ (see Section~\ref{sec:futurework} for more detail) and obtained the following upper bound. As mentioned  in Subsection~\ref{subsec:decomposition}, an \emph{omnian} is a non-leaf vertex whose children are all reticulation vertices. 

\begin{theorem}[Theorem 2.14 in \cite{JettenOLD}]\label{thm:jetten:count}
Let $N$ be a rooted binary tree-based phylogenetic network and  let $B_N$ be its associated bipartite graph with vertex bipartition $\{O, R\}$ and arc-set
$\{(o, r) : o \in O, r \in R, (o, r) \in A(N)\}$, where $O$ and $R$ denote the sets of omnians and reticulations  in $V(N)$, respectively.  Then, 	we have 
	\begin{align*}
		\beta(N)\leq 2^c\cdot \prod_{P\in \mathcal{P} (B_N)}{\frac{1}{2}(|V(P)|+3)}, 
	\end{align*}
where $\beta(N)$ is the number of base trees of $N$, $c$ is the number of cycle components in $B_N$, and $\mathcal{P} (B_N)$ is the set of path components in $B_N$ such that both terminal vertices are in $R$.
\end{theorem}

In \cite{pons}, Pons~\textit{et al.}~discussed the number $\alpha(N)$ of subdivision trees of a tree-based phylogenetic network $N$ in an approach similar to that of Jetten~\cite{JettenOLD}. They obtained an exact formula for $\alpha(N)$ and mentioned that it is possible in polynomial time to count $\alpha(N)$ in an approach focusing on matchings in a bipartite graph associated with $N$   (Theorem 8 in \cite{pons}). Also, they reframed their result as follows, thus showing the connection with the formula in Theorem~\ref{thm:jetten:count}.

\begin{theorem}[Theorem 9 in \cite{pons}]\label{thm:pons:count}
With the same premises and notation as in Theorem~\ref{thm:jetten:count}, we have 	
\begin{align*}
		\alpha(N) = 2^c\cdot \prod_{P\in \mathcal{P} (B_N)}{\frac{1}{2}(|V(P)|+3)}.
	\end{align*}
\end{theorem}

Using Corollary~\ref{cor:formula}, we can straightforwardly understand the formulae in Theorems~\ref{thm:jetten:count} and \ref{thm:pons:count} from a more general perspective. This is because the decomposition expression of $\alpha(N)$ in Corollary~\ref{cor:formula} clarifies what factors constitute the number $\alpha(N)$ for any rooted binary (not necessarily tree-based) phylogenetic $X$-network $N$, in a way reminiscent of the prime factorization of natural numbers. 
When $N$ is a tree-based phylogenetic network, Corollary~\ref{cor:formula} states that    $\alpha(N)$ is determined only by two kinds of factors, namely, the number of crowns and (a half of) the length of each M-fence in the maximal zig-zag trail decomposition $\mathcal{Z}$ of $N$. As can be checked easily, the factor $2^c$ in the formulae in Theorems \ref{thm:jetten:count} and \ref{thm:pons:count}  is exactly the number of admissible subsets for those crowns  as $c$ equals the number of crowns in $\mathcal{Z}$. Also, if we notice that there exists a bijection between the set $\mathcal{P} (B_N)$ and the set of M-fences in $\mathcal{Z}$, we can easily see that  $\prod_{P\in \mathcal{P} (B_N)}{\frac{1}{2}(|V(P)|+3)}$  represents the contribution from the M-fences in $\mathcal{Z}$. Indeed, each element $P$ of $\mathcal{P} (B_N)$  is essentially an M-fence, except that every M-fence has  two more arcs on each end of it that do not appear in $B_N$ (\textit{i.e.}, each M-fence $Z_i$ has four more arcs than its corresponding element $P$ of $\mathcal{P} (B_N)$). Thus, when $Z_i$ is an M-fence, we have $\frac{1}{2}{(|V(P)|-1+4)}=\frac{1}{2}{(|A(P)|+4)}=\frac{1}{2}|A(Z_i)|$.

As noted above, our formula $\alpha(N)=\alpha(Z_1)\times\dots \times \alpha(Z_\ell)$ in Corollary~\ref{cor:formula} holds true for  any rooted binary phylogenetic $X$-network $N$. By virtue of this, the algorithm derived from Corollary~\ref{cor:formula} (Algorithm~\ref{algm:count}) does not require  the input network $N$ to be tree-based, and  can solve Problem~\ref{prob:count} and Problem~\ref{prob:decision} simultaneously. Obviously, we could give an algorithm that immediately returns $\alpha(N)=0$ when $\alpha(Z_i)=0$ is obtained for some $i$, but in Algorithm~\ref{algm:count}, we provide pseudocode that computes the number $\alpha(Z_i)$ of admissible subsets for every $Z_i$ so that the reader can smoothly understand how to use the formula for $\alpha(N)$.

\begin{corollary}\label{cor:counting.complexity}
Algorithm~\ref{algm:count} can solve Problem~\ref{prob:count} in $O(|A(N)|)$ time. 
\end{corollary}
\begin{proof}
	The algorithm counts the number $\alpha(N)$ of subdivision trees of $N$ as follows: 1) compute the maximal zig-zag trail decomposition $\mathcal{Z}=\{Z_1,\dots Z_\ell\}$ of $N$ by Algorithm~\ref{algm}; 2) determine $\alpha(Z_i)$ for each  $Z_i$; 3) compute $\alpha(N)=\alpha(Z_1)\times\dots\times \alpha(Z_\ell)$. The most expensive step is the first preprocessing, which requires $O(|A(N)|)$ time by Proposition~\ref{prop:preprocessing.running.time}. This completes the proof. 
\end{proof}

By Remark~\ref{rem:theta.running.time}, it is guaranteed that Algorithm~\ref{algm:count} is optimal and that Problem~\ref{prob:count} can be solved in $\Theta(|A(N)|)$ time.

\begin{algorithm}
\caption{\textsc{Counting} (Problem~\ref{prob:count})\label{algm:count}}
\begin{algorithmic}
\Require{A rooted binary phylogenetic $X$-network $N$}
\Ensure{The number $\alpha(N)\geq 0$ of subdivision trees of $N$}

\State{call \textsc{Pre-processing} (Algorithm~\ref{algm}) to obtain $\mathcal{Z}=\{Z_1,\dots, Z_\ell\}$}
\State initialize $\alpha_1,\dots, \alpha_\ell := 0$
\ForAll {maximal zig-zag trail $Z_i\in \mathcal{Z}$}

\If{$Z_i$ is a W-fence}
\State $\alpha_i := 0$
\Comment Corollary~\ref{cor:characterization.tbn}
\ElsIf{$Z_i$ is an N-fence}
\State $\alpha_i := 1$
\Comment  $\alpha_i=|\mathcal{S}(Z_i)|$ in the equation~(\ref{eq:counting})
\ElsIf{$Z_i$ is a crown}
\State $\alpha_i := 2$
\Else\; ($Z_i$ is an M-fence)
\State $\alpha_i :=  |A(Z_i)|/2$ 
\EndIf
\EndFor
\State $\alpha(N):=\alpha_1\times \dots \times \alpha_\ell$
\Comment Corollary~\ref{cor:formula}
\State output $\alpha(N)$ and \textbf{halt}
\end{algorithmic}
\end{algorithm}

As we now demonstrate, the number $\alpha(N)$ may give insights into the ``complexity'' of tree-based phylogenetic networks. For example, given a tree-based phylogenetic network $N$ on $X=\{x_1, \dots, x_8\}$ as shown in Figure~\ref{fig:universalupper}, Algorithm~\ref{algm:count} starts by decomposing  $N$ into 21 maximal N-fences consisting of a single arc and 7 maximal M-fences of sizes 2, 4, 6, 8, 10, 12, 14, and so returns $\alpha(N)=7!=5040$.  Comparing this output with the trivial upper bound $2^r=2^{21}=2097152$, where $r$ denotes the number of reticulation vertices of $N$, we can see that it is meaningful to compute  the exact value of $\alpha(N)$. Although the number $\alpha(N)=5040$ may seem huge, it is smaller than the number of rooted binary phylogenetic $X$-trees that is given by $(2|X|-3)!! =13\times 11 \times \cdots \times 5 \times 3\times 1=135135$. In other words, $N$ does not have adequate complexity in order to cover all rooted binary phylogenetic $X$-trees (\textit{i.e.}, all evolutionary scenarios that can be represented using a tree).  Thus, the number $\alpha(N)$ can be used as a quantitative measure for the complexity of $N$, which may have implications on  model selection in evolutionary data analysis.

\begin{figure}[htbp]
\centering
\includegraphics[scale=.66]{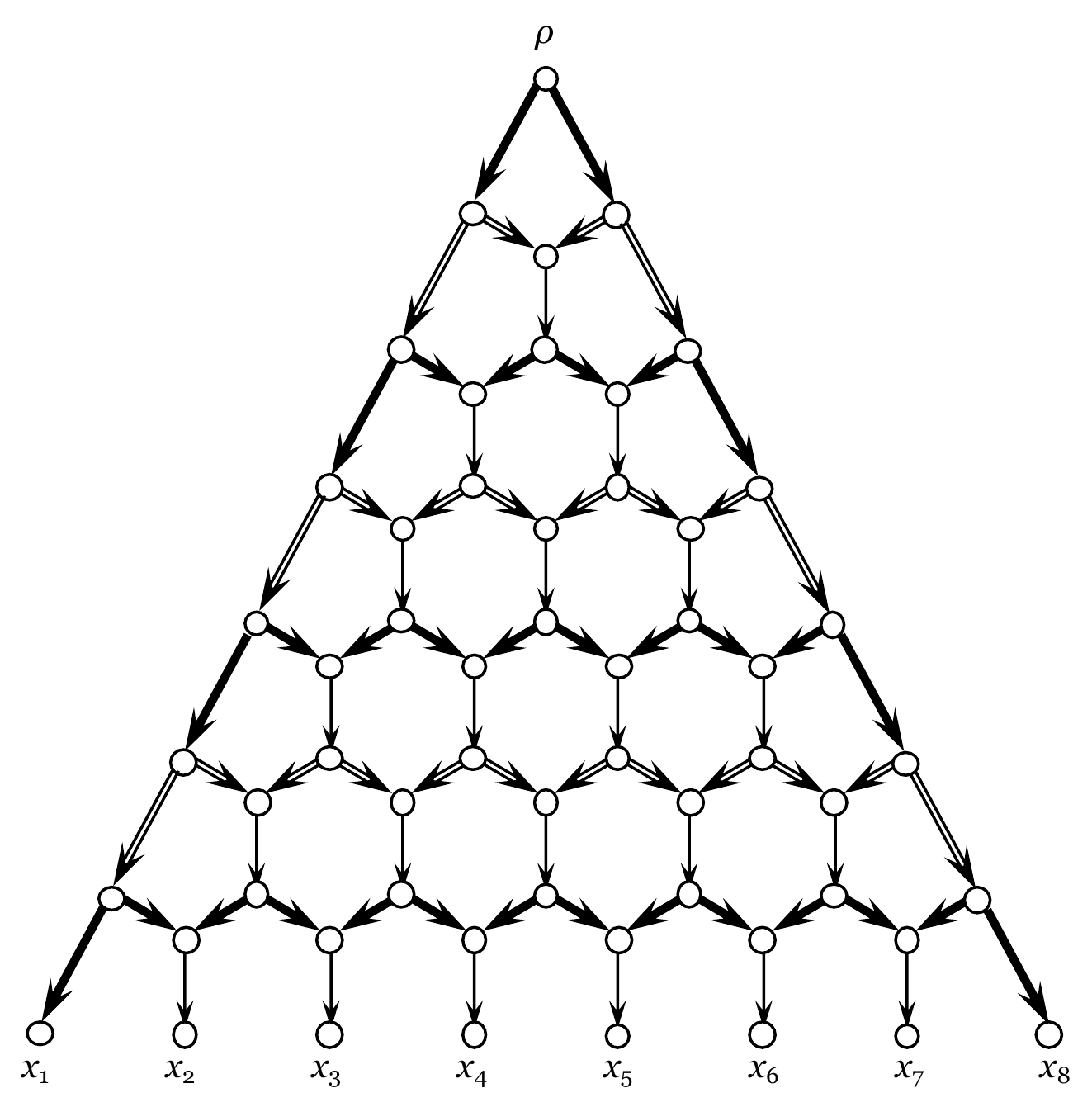}
\caption{
An example to demonstrate our counting algorithm. This network was previously considered in two  independent studies~\cite{UTBN, LX} in the context of constructing ``universal'' tree-based phylogenetic networks.
 \label{fig:universalupper}}
\end{figure}

\subsection{Linear delay algorithm for the enumeration (listing) problem} 

As Theorem~\ref{structure} gives an explicit characterization of the set $\mathcal{T}$ of all subdivision trees of a tree-based phylogenetic network $N$, it also furnishes a straightforward algorithm for solving Problem~\ref{prob:listing}. As in Algorithm~\ref{algm:listing}, it first decomposes $N$ into maximal zig-zag trails $Z_1,\dots, Z_\ell$ and then output each element of $\mathcal{T}=\prod_{i\in [1,\ell]}{\mathcal{S}(Z_i)}$ one after another until it finish listing all.  Recalling Definition~\ref{dfn:polynomial-delay}, we  now prove the following.

\begin{algorithm}
\caption{\textsc{Enumeration}  (Problem~\ref{prob:listing})
 \label{algm:listing}}
\begin{algorithmic}
\Require{A tree-based phylogenetic network $N$ on $X$}
\Ensure{All subdivision trees of $N$}

\State{call \textsc{Pre-processing} (Algorithm~\ref{algm}) to obtain $\mathcal{Z}=\{Z_1,\dots, Z_\ell\}$}
\State update $\mathcal{Z}:=(Z_1,\dots, Z_\ell)$
\Comment fix the ordering of $\mathcal{Z}$ 
\State initialize $\mathcal{A}_1, \dots, \mathcal{A}_\ell, \mathcal{T}:=\emptyset$
\ForAll {maximal zig-zag trail $Z_i\in \mathcal{Z}$}
\If{$Z_i$ is a crown}
\State $\mathcal{S}_i:=\bigl\{ \langle (10)^{|A(Z_i)|/2}\rangle,\, \langle (01)^{|A(Z_i)|/2} \rangle\bigr\}$
\State $\mathcal{A}_i:=$ the family of two subsets of $A(N)$ that are specified by the 1s in  each sequence in $\mathcal{S}_i$
\ElsIf{$Z_i$ is an M-fence}
\State $\mathcal{S}_i:=\bigl\{ \langle 1(01)^p(10)^q 1 \rangle \mid  p,q\in \mathbb{Z}_{\geq 0}, p+q=(|A(Z_i)|-2)/2\bigr\}$
\State $\mathcal{A}_i:=$ the family of $|A(Z_i)|$ subsets of $A(N)$ that are specified by the 1s in  each sequence in $\mathcal{S}_i$
\Else
\State $\mathcal{S}_i:=\bigl\{\langle 1(01)^{(|A(Z_i)|-1)/2}\rangle\bigr\}$
\State $\mathcal{A}_i:=$ the family of the single subset of $A(N)$ that is specified by the 1s in the sequence in $\mathcal{S}_i$
\EndIf
\EndFor
\ForAll {$A\in \prod_{i\in[1,\ell]}{\mathcal{A}_i}$}
\State output $(V(N), A)$ as a subdivision tree of $N$ 
\Comment Theorem~\ref{structure}
\EndFor
\State \textbf{halt}
\end{algorithmic}
\end{algorithm}

\begin{corollary}\label{cor:enumeration}
Algorithm~\ref{algm:listing} can solve Problem~\ref{prob:listing} in $O(|A(N)|)$ time delay. 
\end{corollary}
\begin{proof} 
By Theorem~\ref{structure}, 
the elements of $\mathcal{S}(Z_i)$ can be generated in $O(|A(Z_i)|)$ time if $Z_i\in \mathcal{Z}$ is an M-fence, 
 and   in $O(1)$ time otherwise. Then,  each of the following steps requires $O(|A(N)|)$ time: 
\begin{enumerate}
	\item to output an arbitrary element of $\prod_{i\in [1,\ell]}{\mathcal{S}(Z_i)}$ as $T_1$;
	\item to output a next element of $\prod_{i\in [1,\ell]}{\mathcal{S}(Z_i)}$ that has not been output yet if there exists any and stop otherwise. 
\end{enumerate}
This completes the proof.
\end{proof}

Similarly to  Remark~\ref{rem:theta.running.time}, Algorithm~\ref{algm:listing} is  optimal in terms of time complexity and Problem~\ref{prob:listing} can be solved in $\Theta(|A(N)|)$ time delay. 

In the case when the number $\alpha(N)$ of subdivision trees of $N$ is too large, or when it is not necessary to list all the subdivision trees of $N$, one may abort Algorithm~\ref{algm:listing}  when the designated number $k$ of subdivision trees are generated.  
Recalling that the running time of polynomial delay algorithms is linear in the size of the output (Section~\ref{sec:prob}), we have  the following corollary. 

\begin{corollary}\label{cor:k-listing}
For any tree-based phylogenetic network $N$ on $X$ and for any natural number $k$ with $k \leq \alpha(N)$,  it is possible  to generate $k$  subdivision trees of $N$ in $O(k|A(N)|)$ time.
\end{corollary}

\subsection{Linear time algorithm for the optimization problem}
Let us use the notation $f_i(T)$ to represent $\sum_{a\in {A(T)\cap A(Z_i)}}{w(a)}$ for each maximal zig-zag trail $Z_i$ of a tree-based phylogenetic network $N$. Then, the value of the objective function $f$ for a subdivision tree $T$ can be expressed as  $f(T)=f_1(T)+\dots + f_\ell(T)$. As the choice of an admissible subset of $A(Z_i)$  does not affect the choice of an  admissible subset for any maximal zigzag trail other than $Z_i$, one can get a solution to Problem~\ref{prob:optimization} simply by piecing together an optimal solution within each $Z_i$, \textit{i.e.}, an admissible subset of $A(Z_i)$ to maximize each $f_i$. From this argument, we obtain  Algorithm~\ref{algm:optimization}. 

\begin{corollary}\label{cor:optimisation}
Algorithm~\ref{algm:optimization} can solve Problem~\ref{prob:optimization} in $O(|A(N)|)$ time. 
\end{corollary}
\begin{proof}
Algorithm~\ref{algm:optimization} first decomposes $N$ into maximal zig-zag trails $Z_1,\dots, Z_\ell$, which requires $O(|A(N)|)$ time according to Proposition~\ref{prop:preprocessing.running.time}. For each $Z_i$, it takes $|O(Z_i)|$ time to decide whether $Z_i$ is a crown, M-fence, or N-fence. If $Z_i\in \mathcal{Z}$ is an M-fence, then one can find an optimal solution within $Z_i$  in $O(|A(Z_i)|)$ time,  and otherwise in $O(1)$ time. It takes $O(|A(N)|)$ time to compute the union of optimal solutions within $Z_i$'s and output the resulting subdivision tree $T$. 
 Overall, $O(|A(N)|)$ time suffices. This completes the proof.
\end{proof}

As noted in Remark~\ref{rem:theta.running.time}, Algorithm~\ref{algm:optimization} is optimal and Problem~\ref{prob:optimization} can be solved in $\Theta(|A(N)|)$ time.

\begin{algorithm}
\caption{\textsc{Optimization}  (Problem~\ref{prob:optimization})
 \label{algm:optimization}}
\begin{algorithmic}
\Require{A tree-based phylogenetic network $N$ on $X$ and its associated weighting function $w:A(N)\rightarrow \mathbb{R}_{\geq 0}$}
\Ensure{A subdivision tree $T$ of $N$ to maximize  $\sum_{a\in A(T)} {w(a)}$}

\State{call \textsc{Pre-processing} (Algorithm~\ref{algm}) to obtain $\mathcal{Z}=\{Z_1,\dots, Z_\ell\}$}
\State update $\mathcal{Z}:=(Z_1,\dots, Z_\ell)$
\State initialize $\mathcal{A}^*_1, \dots, \mathcal{A}^*_\ell:=\emptyset$
\ForAll {maximal zig-zag trail $Z_i\in \mathcal{Z}$}
\If{$Z_i$ is a crown}
\State $\mathcal{S}_i:=\bigl\{ \langle (10)^{|A(Z_i)|/2}\rangle,\, \langle (01)^{|A(Z_i)|/2} \rangle\bigr\}$
\State $\mathcal{A}_i:=$ the family of two subsets of $A(N)$ that are specified by the 1s in  each sequence in $\mathcal{S}_i$
\State $A^*_i :=$ an element of $\arg \max_{A\in \mathcal{A}_i}{\sum_{a\in A}{w(a)}}$ 
\ElsIf{$Z_i$ is an M-fence}
\State $\mathcal{S}_i:=\bigl\{ \langle 1(01)^p(10)^q 1 \rangle \mid  p,q\in \mathbb{Z}_{\geq 0}, p+q=(|A(Z_i)|-2)/2\bigr\}$
\State $\mathcal{A}_i:=$ the family of $|A(Z_i)|$ subsets of $A(N)$ that are specified by the 1s in  each sequence in $\mathcal{S}_i$
\State $A^*_i :=$ an element of $\arg \max_{A\in \mathcal{A}_i}{\sum_{a\in A}{w(a)}}$ 
\Else 
\State $A^*_i :=$ the subset of $A(N)$ that are specified by the 1s in the sequence $\langle 1(01)^{(|A(Z_i)|-1)/2}\rangle$
\EndIf
\EndFor
\State $T:=(V(N), A^*_1\cup\dots\cup A^*_\ell)$
\Comment The union of optimal solutions within $Z_i$'s
\State output $T$ and \textbf{halt}
\end{algorithmic}
\end{algorithm}

\section{Remark on a special class of non-binary phylogenetic networks}\label{sec:non-binary}
In recent years, many studies have discussed tree-based phylogenetic networks that are not necessarily binary and it has been shown that the results established in binary settings may or may not hold in the non-binary case  (\textit{e.g.},~\cite{fischer2018non, mathbio2018, JettenOLD, Jetten, pons}). For example, in \cite{Jetten}, Jetten and van Iersel pointed out that  Theorem~\ref{thm:jetten.characterization} (and thus also Theorem~\ref{thm:zhang}) does  not hold in general when $N$ is not binary but showed that,  if a slightly modified bipartite graph is used,  it is similarly possible to obtain both a matching-based characterization of non-binary tree-based networks  (Theorem 3.4 in \cite{Jetten}) and an algorithm that can decide in polynomial time whether or not a given non-binary phylogenetic network is tree-based by computing a maximum-sized matching in a bipartite graph (Corollary 3.5 in\cite{Jetten}). Also, Pons~\textit{et~al.}~\cite{pons} showed the results of Francis~\textit{et~al.}~\cite{newchara2018} mentioned in Remark~\ref{rem:deviation.previous.method} still hold true if $N$ is non-binary (Theorems 5 and 6 in \cite{pons}) and indicated that the non-binary version of the deviation quantification problem (Problem~\ref{prob:deviation}) can be solved  using essentially the same $O(|V(N)|^{3/2})$  time algorithm as proposed in \cite{newchara2018}. However, in \cite{pons}, it was  left as an open question about how to count the number $\alpha(N)$ of subdivision trees when $N$ is not necessarily binary.

Therefore, we point out that all results and algorithms in this paper (and also the relevant results in \cite{JettenOLD, Jetten, pons, LX}) hold true for a special class of non-binary phylogenetic networks. To explain this, we now slightly generalize the notion of rooted binary phylogenetic networks to ``almost binary'' ones.  The definition of rooted \emph{almost-binary} phylogenetic $X$-networks is the same as Definition~\ref{dfn:rbpn}, except that the third condition imposed on each vertex $v\in V(N)\setminus (X\cup \{\rho\})$ is relaxed from satisfying $\{{\it deg}^-_N(v), {\it deg}^+_N(v)\}= \{1,2\}$ to satisfying ${\it deg}^-_N(v)\leq 2$ and ${\it deg}^+_N(v)\leq 2$.  In other words, almost-binary phylogenetic networks can contain such vertices as shown in Figure~\ref{fig:X}. An example of a tree-based phylogenetic network that is almost-binary is illustrated in Figure~\ref{fig:gridnetwork}. 

One can adapt the definitions of subdivision trees and base trees (Definition~\ref{dfn:subdiv}) and of tree-based phylogenetic networks (Definition~\ref{dfn:tbn}) to the almost-binary case as they are, where we note that if $N$ is a tree-based phylogenetic network on $X$ that is almost-binary, then any base tree of $N$ is necessarily binary. The characterization of tree-based networks in terms of the existence of an admissible subset of $A(N)$ (Theorem~\ref{thm:bijection}) is still valid in the almost-binary case. 

A maximal zig-zag trail in a rooted almost-binary phylogenetic $X$-network $N$ is also defined all the same as before.  
We note that if $N$ is a rooted almost-binary phylogenetic $X$-network, then any two maximal zig-zag trails in $N$ must be arc-disjoint because ${\it deg}^-_N(v)\leq 2$ and ${\it deg}^+_N(v)\leq 2$ hold. This means that any rooted almost-binary phylogenetic $X$-network can be canonically decomposed into its unique maximal zig-zag trails as stated in Theorem~\ref{uniquely.decomposable}. As the results and algorithms in this paper are consequences of Theorem~\ref{uniquely.decomposable},  we know that they should be   correct if $N$ is not binary but almost-binary. 
It follows that many previous results (\textit{e.g.}, Theorem~\ref{thm:zhang}, Theorem~\ref{thm:jetten.characterization}, the decision algorithm in \cite{LX}, the formulae for $\alpha(N)$ in \cite{JettenOLD, pons}, and the counting algorithm in \cite{pons}) can be correctly applied to such a non-binary network as described in Figure~\ref{fig:gridnetwork}. It is worth remembering that some non-binary phylogenetic networks are just as tractable as the binary ones because problems may turn out to be easier to solve than anticipated. Indeed, in the case when $N$ is almost-binary, Algorithm~\ref{algm:deviation.quantification} can solve Problem~\ref{prob:deviation} in  $O(|A(N)|)$ time in the same way as if $N$ is binary, which is better than the time bound  $O(|V(N)|^{3/2})$ in \cite{pons}.

\begin{figure}[htbp]
\centering
\includegraphics[scale=.8]{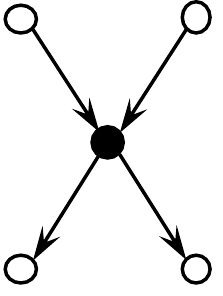}
\caption{An illustration of a vertex that is allowed to exist in rooted ``almost-binary'' phylogenetic $X$-networks (see the text for the definition).
\label{fig:X}}
\end{figure}

\begin{figure}[htbp]
\centering
\includegraphics[scale=.65]{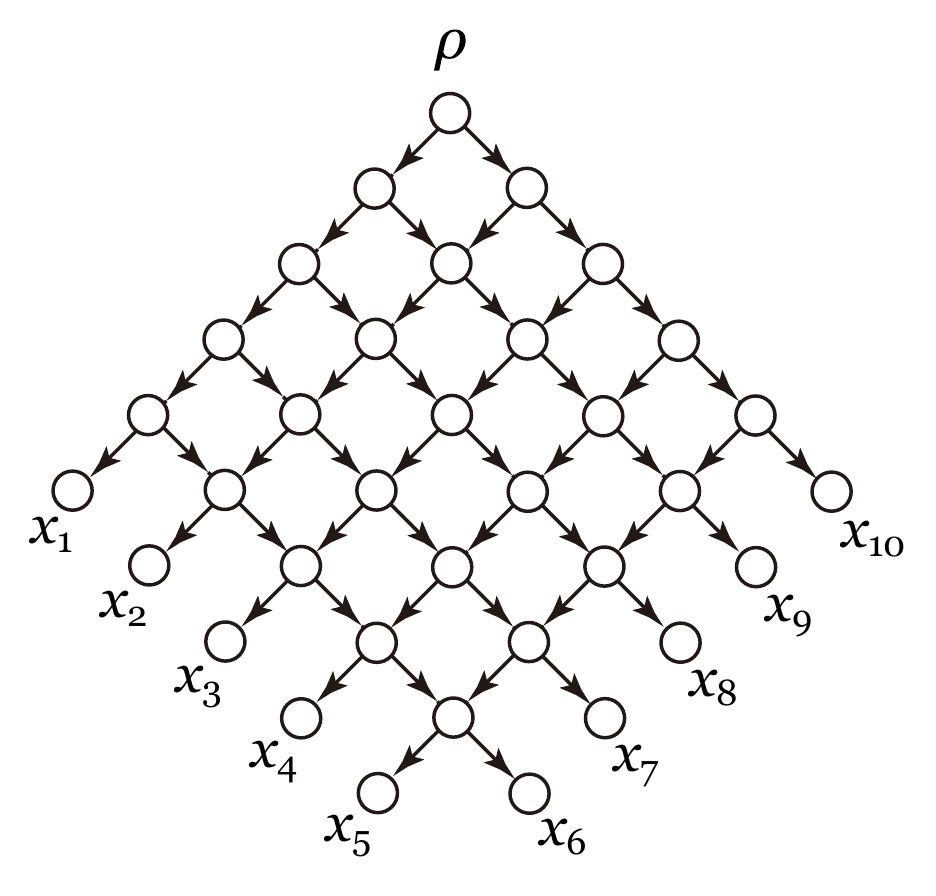}
\caption{An example of a rooted almost-binary phylogenetic $X$-network, which is tree-based. 
\label{fig:gridnetwork}}
\end{figure}

\section{Conclusion and further research directions}\label{sec:futurework}
The contributions of the paper are summarized as follows. We proved the structure theorem (Theorem~\ref{uniquely.decomposable}) that gives a way to canonically decompose any rooted binary phylogenetic network into its substructures called maximal zig-zag trails that are uniquely determined. This theorem has considerable implications for tree-based phylogenetic networks because it does not only give decomposition-based characterizations of tree-based phylogenetic networks (Lemma~\ref{iff} and Corollary~\ref{cor:characterization.tbn}) and a new cohesive perspective to proving some known results, but also, more importantly, leads to a characterization of the set of subdivision trees of a tree-based phylogenetic network in the form of a direct product of families of admissible arc-sets within each substructure (Theorem~\ref{structure}), which is in the spirit of the structure theorem for finitely generated Abelian groups and many other structural results.  
Moreover, from the above main results, we derived a series of linear time (and linear time delay) algorithms for solving a variety of old or new computational problems on tree-based phylogenetic networks and subdivision trees (Problems~\ref{prob:decision}, \ref{prob:deviation}, \ref{prob:count}, \ref{prob:listing}, and \ref{prob:optimization}). In other words, we obtained numerous corollaries and made it possible to do efficiently all of the following: to decompose a given rooted binary phylogenetic network $N$ into its maximal zig-zag trails (Algorithm~\ref{algm});  to decide whether or not $N$ is tree-based and find a subdivision tree of $N$ if $N$ is tree-based (Algorithm~\ref{algm:decision.search}); to measure the deviation of $N$ from being tree-based (Algorithm~\ref{algm:deviation.quantification}); to count the number of subdivision trees of $N$ (Algorithm~\ref{algm:count}); to list all subdivision trees of $N$ (Algorithm~\ref{algm:listing}), and to compute an optimal subdivision tree of $N$ that maximizes or minimizes a prescribed objective function (Algorithm~\ref{algm:optimization}). Each of the above algorithms is optimal in terms of time complexity (Remark~\ref{rem:theta.running.time}).  Our results do not only answer many questions or unify and extend various results in the relevant literature, but also open new possibilities for statistical applications of tree-based phylogenetic networks, such as generating subdivision trees uniformly at random and computing a maximum likelihood subdivision tree, which have not been considered previously. 

Finally, we end this paper by mentioning some open problems and possible directions for future research that would be interesting to pursue.

\subsection{Expanding the  application scope of tree-based networks and subdivision trees}
While this paper has focused on  the five fundamental problems for tree based phylogenetic networks and subdivision trees, our results suggest further avenues of research beyond these five problems. Indeed, because Theorem~\ref{structure} gives an explicit characterization of the set of subdivision trees of a tree-based phylogenetic network, it would be quite possible to design polynomial time algorithms that solve more advanced or realistic computational problems than those considered here. For example, Hayamizu and Makino \cite{topkranking}  formulate a ``top-$k$ ranking problem'', which combines the listing and optimization problems on subdivision trees, and provide a linear time delay algorithm for solving it. It would be interesting  to explore other biologically meaningful computational problems for which efficient algorithms can be developed, and such research would contribute to expanding the range of applications of tree-based phylogenetic tree networks.

\subsection{Related but different counting problems}\label{subsec:direction1}
Let us recall that a tree-based phylogenetic network can have some isomorphic subdivision trees (see Figure~\ref{fig:non-isomorphic}) and that subdivision trees are different from base trees (Definition~\ref{dfn:subdiv}). While we conjecture that the following Problem~\ref{prob:basetree} is \#P-complete  in contrast to Problem~\ref{prob:count} being solvable in linear time, it is still unknown whether the following two problems can be solved in polynomial time.

\begin{problem}\label{prob:non-isomorphic.subdivision.tree}
Given a rooted binary phylogenetic $X$-network $N$, count the number $\widehat{\alpha}(N)$ ($\geq 0$) of non-isomorphic subdivision trees of $N$.
\end{problem}

\begin{problem}[\cite{FS}]\label{prob:basetree}
Given a rooted binary phylogenetic $X$-network $N$, count the number $\beta(N)$ ($\geq 0$) of base trees of $N$.
\end{problem}

Note that $\alpha(N)\geq \widehat{\alpha}(N) \geq \beta(N)$ holds by definition. 
For instructive purposes,  we now demonstrate the differences between them with two examples. 
Suppose $N_1$ is the network in Figure~\ref{fig:crossovers}. If this $N_1$ is input, then Algorithm~\ref{algm:count} returns  $\alpha(N_1)=2^2$; however, we can easily see that $\widehat{\alpha}(N_1)=\beta(N_1)=1$ holds as $N_1$ virtually contains only one phylogenetic tree on $X=\{x_1, x_2\}$. Next, we assume that $N_2$ is the network shown in Figure~\ref{fig:non-isomorphic}. Then, we  have $\alpha(N_2)=6$ by  Algorithm~\ref{algm:count}. By examining each element of the set $\mathcal{T}$ of subdivision trees of $N_2$ that is produced by Algorithm~\ref{algm:listing}, we see that the elements of $\mathcal{T}$ are are all distinct, so $\widehat{\alpha}(N_2)=6$ holds. However,  $N_2$ has exactly two subdivision trees that are embeddings of the same rooted binary phylogenetic tree (highlighted in bold in Figure~\ref{fig:non-isomorphic}), and thus $\beta(N_2)=5$ holds.

Besides the complexity of the above two problems, it would be also meaningful to study the relationship between $\alpha(N)$, $\widehat{\alpha}(N)$, and $\beta(N)$ towards the development of useful criteria for analyzing the  complexity of phylogenetic networks.   For example, Francis and Moulton \cite{FM2018} obtained a result meaning that $\alpha(N)=\widehat{\alpha}(N)$ holds if $N$ is a tree-child network (\textit{i.e.}, a tree-based phylogenetic network with the special property such that each non-leaf vertex of N has at least one child that is not a reticulation vertex) (Theorem~3.3 in \cite{FM2018}), but  we note that the converse does not hold. In fact, the subdivision trees of the network in Figure~\ref{fig:non-isomorphic} are all distinct although it is not tree-child. Then, what is a necessary and sufficient condition for $\alpha(N)=\widehat{\alpha}(N)$, and what about  $\widehat{\alpha}(N)=\beta(N)$?

\begin{figure}[htbp]
\centering
\includegraphics[scale=.6]{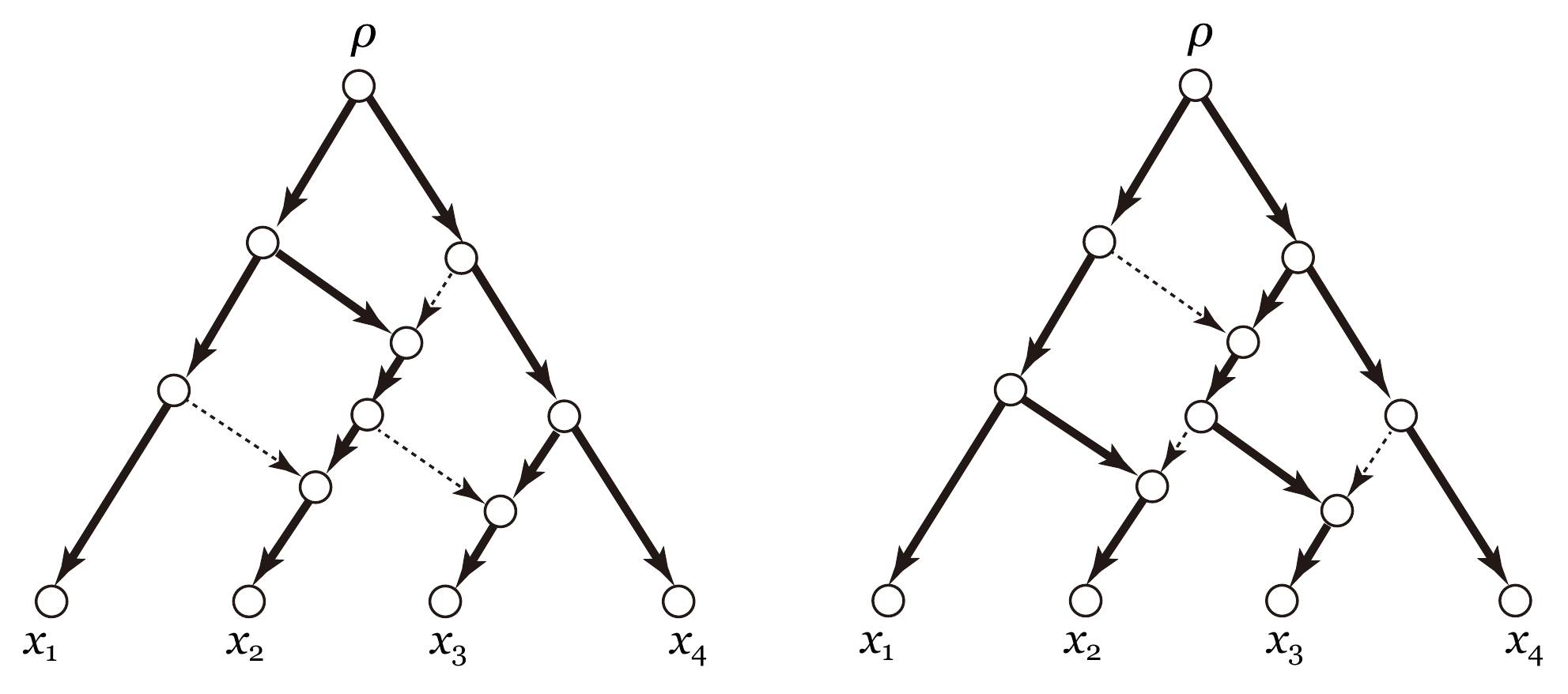}
\caption{An illustration of two non-isomorphic subdivision trees that are embeddings of the same phylogenetic tree, which demonstrates the difference between Problem~\ref{prob:non-isomorphic.subdivision.tree} and Problem~\ref{prob:basetree} (see Subsection~\ref{subsec:direction1} for more detail).
\label{fig:non-isomorphic}}
\end{figure}

\subsection{Structure of unrooted phylogenetic networks}
The notion of unrooted (undirected) tree-based phylogenetic networks were originally defined in \cite{unrootedTBN2018} and has been discussed in several studies in recent years (\textit{e.g.}, \cite{FF2020, fischer2018non, mathbio2018}). However, in contrast to the rooted case, there are many computational difficulties in this direction even in binary settings. For example, given an unrooted binary phylogenetic network $N$, the problem of deciding whether or not $N$ is tree-based  is NP-complete (Theorem~2 in \cite{unrootedTBN2018}). This implies that, as Fischer and Francis \cite{FF2020} pointed out, all of the indices for measuring deviations from being tree-based proposed in \cite{FF2020} are NP-hard to compute.
Considering that phylogenetic networks reconstructed using distance data are  necessarily unrooted, it would be important to explore subclasses of unrooted phylogenetic networks with a mathematically nice structure in order to develop new and useful methods for analyzing biological data.

\bibliographystyle{amsplain}

\bibliography{references.bib}

\section*{Acknowledgment}
This study was supported by JST PRESTO Grant Numbers JPMJPR16EB and JPMJPR1929. The author thanks the organizers of the Portobello 2018 Conference (The Interface of Mathematics and Biology, The 22nd Annual New Zealand Phylogenomics Meeting) where she announced most results in this paper in her talk \cite{portobello2018}.
The author is also grateful to the anonymous reviewers for their quality comments that have greatly improved the readability of this paper, to Kazuhisa Makino for suggesting Problem~\ref{prob:optimization} and Section~\ref{sec:non-binary} and for many other helpful comments, to Mike Steel for some editorial suggestions and for useful discussion on Corollary~\ref{thm:deviation.equals.W-fences}, and to Andrew Francis, Leo van Iersel, and Louxin Zhang for providing information on relevant references.

\end{document}